\documentclass[a4paper,11pt]{amsart}
\usepackage{mathrsfs}
\usepackage{cases}
\usepackage{epic}
\usepackage{amsfonts}
\usepackage{graphicx}
\usepackage{amsmath}
\usepackage{amssymb, upgreek}
\usepackage{bm}
\usepackage{latexsym,todonotes}
\usepackage{pdflscape}
\usepackage[all]{xypic}
\usepackage{ytableau}
\usepackage[enableskew,vcentermath]{youngtab}
\usepackage[all]{xy}
\usepackage{color}
\usepackage{colordvi}
\usepackage{multicol}
\usepackage[linktocpage=true]{hyperref}
\usepackage{dsfont}
\hypersetup{colorlinks,linkcolor=blue,urlcolor=cyan,citecolor=blue}
\newcommand{\bvs}{\mathbf{\varsigma}}
\newcommand{\vs}{\varsigma}

\allowdisplaybreaks

\begin{document}
\input xy
\xyoption{all}

\newtheorem{innercustomthm}{{\bf Main~Theorem}}
\newenvironment{customthm}[1]
  {\renewcommand\theinnercustomthm{#1}\innercustomthm}
  {\endinnercustomthm}

  \def\haTh{\widehat{\Theta}}
\def \haH{\widehat{H}}
\newcommand{\bB}{{\mathbf B }}
\newcommand{\bDel}{\boldsymbol{\Delta}}
\newcommand{\bBKH}{\acute{\mathbf H}}
\newcommand{\bH}{\mathbf H}
\newcommand{\BKH}{\acute{H}}
\def \tM{\mathcal{M}\widetilde{\ch}}
\def \E{K}
\def\tTT{\mathrm T}
\def \hath{\widehat{\theta}}
\def \bt{\mathbf t}
\def \bn{\mathbf n}
\def \bh{\mathbf h}
\newcommand{\cc}{{\mathcal C}}
\def \bC{\mathbf C}
\def \dB{\Theta}
\def \bpi{\boldsymbol{\pi}}
\def \dt{{\dot \theta}}
\def \dT{{\dot \Theta}}
\def \co{\mathcal O}
\def \bJ{{\mathbf J}}
\def \ch{{\mathcal H}}
\def \cm{{\mathcal M}}
\def \ct{{\mathcal T}}
\def \bS{\mathbf S}
\def \Uto{\U_{\mathrm{tor}}}
\def \tMH{{\cm\widetilde{\ch}(\Lambda^\imath)}}
\def \tMHl{{\cm\widetilde{\ch}(\bs_\ell\Lambda^\imath)}}
\def \haB{\widehat{B}}
\newcommand{\tMHX}{{}^\imath\widetilde{\ch}(\X_\bfk)}
\newcommand{\tCMH}{{}^\imath\widetilde{\cc}(\bfk Q)}
\newcommand{\tCMHX}{{}^\imath\widetilde{\cc}(\X_\bfk)}
\renewcommand{\mod}{\operatorname{mod}\nolimits}
\newcommand{\tCMHC}{{}^\imath\widetilde{\cc}(\bfk C_n)}
\newcommand{\tCHX}{\widetilde{\cc}(\X_\bfk)}

\numberwithin{equation}{section}

\renewcommand{\ker}{\operatorname{Ker}\nolimits}

\def \tH{{\widetilde{H}}}
\def \tTH{\widetilde{\Theta}}
\def \cp{\mathcal P}
\newcommand{\Hom}{\operatorname{Hom}\nolimits}
\newcommand{\RHom}{\operatorname{RHom}\nolimits}

\newcommand{\Aut}{\operatorname{Aut}\nolimits}
\newcommand{\Id}{\operatorname{Id}\nolimits}
\newcommand{\dimv}{\operatorname{\underline{dim}}\nolimits}
\newcommand{\Sym}{\operatorname{Sym}\nolimits}
\newcommand{\Ext}{\operatorname{Ext}\nolimits}
\newcommand{\Fac}{\operatorname{Fac}\nolimits}
\newcommand{\add}{\operatorname{add}\nolimits}

\def \y{{B}}
\def \haB{\widehat{B}}

\def \bTH { \boldsymbol{\Theta}}
\def \bDel{ \boldsymbol{\Delta}}

\newcommand{\mbf}{\mathbf}
\newcommand{\mbb}{\mathbb}
\newcommand{\mrm}{\mathrm}
\newcommand{\A}{\mathcal A}
\newcommand{\cbinom}[2]{\left\{ \begin{matrix} #1\\#2 \end{matrix} \right\}}
\newcommand{\dvev}[1]{{B_1|}_{\ev}^{{(#1)}}}
\newcommand{\dv}[1]{{B_1|}_{\odd}^{{(#1)}}}
\newcommand{\dvd}[1]{t_{\odd}^{{(#1)}}}
\newcommand{\dvp}[1]{t_{\ev}^{{(#1)}}}
\newcommand{\ev}{\bar{0}}
\newcommand{\End}{\mrm{End}}
\newcommand{\rank}{\mrm{rank}}
\newcommand{\de}{\delta}
\def \C{{\mathbb C}}
\newcommand{\kk}{\widetilde{\mathbf{K}}}
\newcommand{\la}{\lambda}
\newcommand{\LR}[2]{\left\llbracket \begin{matrix} #1\\#2 \end{matrix} \right\rrbracket}
\newcommand{\N}{\mathbb N}
\newcommand{\bbZ}{\mathbb Z}
\newcommand{\odd}{\bar{1}}
\newcommand{\one}{\mathbf 1}
\newcommand{\ov}{\overline}
\newcommand{\qbinom}[2]{\begin{bmatrix} #1\\#2 \end{bmatrix} }
\newcommand{\Q}{\mathbb Q}
\newcommand{\sll}{\mathfrak{sl}}
\newcommand{\T}{\texttt{\rm T}}
\newcommand{\ttt}{\mathfrak{t}}
\newcommand{\U}{\mbf U}
\newcommand{\K}{\mathbb K}
\newcommand{\F}{\mathbb F}
\newcommand{\bi}{\imath}
\newcommand{\bs}{\mathbf s}
\newcommand{\arxiv}[1]{\href{http://arxiv.org/abs/#1}{\tt arXiv:\nolinkurl{#1}}}
\newcommand{\Udot}{\dot{\mbf U}}
\newcommand{\UA}{{}_\A{\mbf U}}
\newcommand{\UAdot}{{}_\A{\dot{\mbf U}}}
\newcommand{\Ui}{{\mbf U}^\imath}
\newcommand{\Uj}{{\mbf U}^\jmath}
\newcommand{\vev}{v^+_{2\la} }
\newcommand{\vodd}{v^+_{2\la+1} }
\newcommand{\Y}{\bB}
\newcommand{\Z}{\mathbb Z}
\newcommand{\Pa}{\operatorname{Pa}\nolimits}
\newcommand{\TT}{\mathbf T}
\newcommand{\B}{\mbf V}
\newcommand{\D}{\mbf D}
\newcommand{\BA}{{}_\A{\B}}
\newcommand{\DA}{{}_\A{\B}'}
\def \X{\mathbb X}
\def \fg{\mathfrak{g}}
\def \bU{{\mathbf U}}
\newcommand{\tK}{\widetilde{K}}
\def \I{\mathbb{I}}
\def \bv{v}
\def \cv{\mathcal V}
\def \cu{\mathcal U}
\def \LaC{\Lambda_{\texttt{Can}}}
\newcommand{\tUiD}{{}^{\text{Dr}}\tUi}
\def \nua{a}
\def \hn{\widehat{\mathfrak{n}}}

\newcommand{\UU}{{\mathbf U}\otimes {\mathbf U}}
\newcommand{\UUi}{(\UU)^\imath}
\newcommand{\tUU}{{\tU}\otimes {\tU}}
\newcommand{\tUUi}{(\tUU)^\imath}
\newcommand{\tUi}{\widetilde{{\mathbf U}}^\imath}
\newcommand{\sqq}{{\bf v}}
\newcommand{\sqvs}{\sqrt{\vs}}
\newcommand{\dbl}{\operatorname{dbl}\nolimits}
\newcommand{\swa}{\operatorname{swap}\nolimits}
\newcommand{\Gp}{\operatorname{Gp}\nolimits}
\newcommand{\coker}{\operatorname{Coker}\nolimits}

\newcommand{\tU}{\widetilde{\mathbf U}}

\def \btau{{{\tau}}}
\newcommand{\tk}{\Bbbk}

\def \ff{B}
\def \cJ{\mathcal{J}}

\def \fn{\mathfrak{n}}
\def \fh{\mathfrak{h}}
\def \fu{\mathfrak{u}}
\def \fv{\mathfrak{v}}
\def \fa{\mathfrak{a}}
\def \fk{\mathfrak{k}}

\def \cl{L}

\def \tf{\widetilde{f}}

\def \K{\mathbb{K}}
\def \R{\mathbb{R}}
\def \SS{\mathbb{S}}

\def \BF{\digamma}
\def \BG{\mathbb G}
\def \tMHL{{}^{\imath}\widetilde{\ch}(\LaC^{\imath,op})}
\def \cc{\mathcal C}
\def\ca{\mathcal A}
\newcommand{\Iso}{\operatorname{Iso}\nolimits}
\renewcommand{\Im}{\operatorname{Im}\nolimits}
\newcommand{\res}{\operatorname{res}\nolimits}
\newcommand{\iH}{{}^\imath\widetilde{\ch}}
\newcommand{\Mod}{\operatorname{Mod}\nolimits}
\newcommand{\coh}{\operatorname{coh}\nolimits}
\newcommand{\rep}{\operatorname{rep}\nolimits}
\newcommand{\Ker}{\operatorname{Ker}\nolimits}
\newcommand{\tUiDgr}{{\text{gr}}\tUiD}

\def \G{\mathbb G}
\def \PL{\mathbb{P}^1_{\bfk}}
\def \scrM{\mathscr M}
\def \scrf{\mathscr F}
\def \scrt{\mathscr T}
\def \P{\mathbb P}
\def \cI{\mathcal I}
\def \cJ{\mathcal J}
\def \cs{{\mathcal{S}}}
\def \ch{\mathcal H}
\def \cd{\mathcal D}
\def\bfk{\mathbf{k}}
\def \ck{\mathcal K}
\def \bp{\mathbf p}
\def \ul{\underline}
\def \fp{\mathfrak p}
\def \QJ{Q_{\texttt{J}}}
\def \II{\I_0}
\def \Lg{L\fg}

\def \bvs{{\boldsymbol{\varsigma}}}
\newcommand{\ci}{{\I}_{\btau}}
\def \cn{\mathcal N}

\def\tor{{\rm tor}}
\newcommand{\calc}{{\mathcal C}}
\newcommand{\haT}{\widehat{\Theta}}

\def \La{\Lambda}
\def \iLa{\Lambda^\imath}
\def \BH{\mathbb H}
\def \bH{\mathbf{H}}
\newcommand{\gr}{\operatorname{gr}\nolimits}
\newcommand{\wt}{\text{wt}}
\def \cR{\mathcal R}
\def \ce{\mathcal E}
\def \cb{\mathcal B}
\def \bla{\boldsymbol{\lambda}}
\def \blx{x}

\newtheorem{theorem}{Theorem}[section]
\newtheorem{acknowledgement}[theorem]{Acknowledgement}
\newtheorem{lemma}[theorem]{Lemma}
\newtheorem{proposition}[theorem]{Proposition}
\newtheorem{corollary}[theorem]{Corollary}
\newtheorem{remark}[theorem]{Remark}
\newtheorem{definition}[theorem]{Definition}
\newtheorem*{thm}{Theorem}
\numberwithin{equation}{section}

\title[New realization of $\imath$quantum groups via $\Delta$-Hall algebras]{New realization of $\imath$quantum groups via $\Delta$-Hall algebras}

\author[Jiayi Chen]{Jiayi Chen}
\address{ School of Mathematical Sciences, Xiamen University, Xiamen 361005, P.R. China}
\email{jiayichen.xmu@foxmail.com}

\author[Yanan Lin]{Yanan Lin}
\address{ School of Mathematical Sciences, Xiamen University, Xiamen 361005, P.R. China}
\email{ynlin@xmu.edu.cn}

\author[Shiquan Ruan]{Shiquan Ruan$^*$}
\address{ School of Mathematical Sciences, Xiamen University, Xiamen 361005, P.R. China}
\email{sqruan@xmu.edu.cn}

\thanks{$^*$ the corresponding author}

\begin{abstract} For an essentially small hereditary abelian category $\mathcal{A}$, we define a new kind of algebra $\mathcal{H}_{\Delta}(\mathcal{A})$, called the $\Delta$-Hall algebra of $\mathcal{A}$. The basis of $\mathcal{H}_{\Delta}(\mathcal{A})$ is the isomorphism classes of objects in $\mathcal{A}$, and the $\Delta$-Hall numbers calculate certain three-cycles of exact sequences in $\mathcal{A}$. We show that the $\Delta$-Hall algebra $\mathcal{H}_{\Delta}(\mathcal{A})$ is isomorphic to the 1-periodic derived Hall algebra of $\mathcal{A}$. By taking suitable extension and twisting, we can obtain the $\imath$Hall algebra and the semi-derived Hall algebra associated to $\mathcal{A}$ respectively.
	
	When applied to the the nilpotent representation category $\mathcal{A}={\rm rep^{nil}}(\mathbf{k} Q)$ for an arbitrary quiver $Q$ without loops, the (\emph{resp.} extended) $\Delta$-Hall algebra provides a new realization of the (\emph{resp.} universal) $\imath$quantum group associated to $Q$.
	
\end{abstract}

\maketitle

 \setcounter{tocdepth}{2}

\section{Introduction}

The Hall algebra approach to quantum groups is a hot topic in representation theory of algebras.
In 1990, Ringel \cite{R1} constructed a Hall algebra associated to a quiver $Q$ over a finite field, and identified its generic version with the positive part of quantum groups when $Q$ is Dynkin type.
Later, Green \cite{G} generalized it to Kac-Moody setting.

Since then, people made efforts to the realization of the whole quantum groups via Hall algebras. Xiao \cite{X97} constructed the Drinfeld double of Hall algebras. Bridgeland \cite{Br13} considered the Hall algebra of 2-periodic complexes of projective modules for a hereditary algebra.
Lu-Peng \cite{LP16} generalized Bridgeland's construction to the
hereditary abelian categories (perhaps without projective objects).
The Hall type Lie algebras of root categories can also be used to realize Kac-Moody algebras, see Peng-Xiao \cite{PX00} and Lin-Peng \cite{LP3}.

Ringel's version of Hall algebra has found further generalizations and improvements
which allow more flexibilities.
To\"{e}n \cite{T06} and Xiao-Xu \cite{XX08} introduced the derived Hall algebras for
triangulated categories satisfying certain homological finiteness conditions.
Xu-Chen \cite{XC13} constructed derived Hall algebras for odd periodic triangulated categories.
Gorsky \cite{Gor18} investigated semi-derived Hall algebras for Frobenius categories.
Lu defined semi-derived Hall algebras for the so-called weakly 1-Gorenstein categories in \cite[Appendix A]{LW19a}.


The $\imath$quantum group is a generalization of quantum groups, which arises from the construction of quantum symmetric pairs by Letzter \cite{Let99}; see \cite{Ko14} for an extension to Kac-Moody type.
A striking breakthrough of the $\imath$quantum groups is the discovery of canonical basis by Bao-Wang \cite{BW18a}. As outlined by Bao-Wang, most of the fundamental constructions of quantum groups should admit generalizations in the setting of $\imath$quantum groups, see \cite{BK19, BW18a, BW18b} for generalizations of (quasi) R-matrix and canonical bases, and also see \cite{BKLW18, Li19, LW21b} for geometric realizations and \cite{BSWW18} for KLR type categorification.

In recent years, Lu-Wang \cite{LW19a, LW20a} have developed $\imath$Hall algebras of $\imath$quivers to realize the universal quasi-split $\imath$quantum groups of Kac-Moody type.
Let's briefly recall the construction of $\imath$Hall algebras. For a hereditary abelian category $\ca$ over a finite field $\F_q$, we have the Ringel-Hall algebra $\mathcal{H}(\mathcal{C}_1(\ca))$ of the category $\mathcal{C}_1(\ca)$ of 1-periodic complexes over $\ca$, and the $\imath$Hall algebra $\iH(\mathcal{A})$ is then obtained from $\mathcal{H}(\mathcal{C}_1(\ca))$ by successively applying quotient, localization and twisting. The construction is difficult, but it turns out that the  $\imath$Hall algebra accesses a nice basis and a simplified multiplication formula, c.f. \cite{LW20a}.

Recall that the $\imath$quantum group is a quotient of the universal $\imath$quantum group by a certain ideal generated by central elements.
It has been proved in \cite{CLR} that the 1-periodic derived Hall algebra of $\ca$ is a quotient algebra of the $\imath$Hall algebra $\iH(\mathcal{A})$, which produces a realization for the split $\imath$quantum group.


%
%
%

The goal of this paper is to construct new algebras in a more direct way to realize (universal) $\imath$quantum groups. Namely,
we define the \emph{$\Delta$-Hall algebra} $\mathcal{H}_{\Delta}(\mathcal{A})$ associated to $\ca$. Following Ringel's original construction, the basis of $\mathcal{H}_{\Delta}(\mathcal{A})$ is the isomorphism classes of objects in $\ca$, while the Hall number is replaced by the so-called $\Delta$-Hall number. Roughly speaking, the $\Delta$-Hall number $\widehat{F}^M_{AB}$ for $A,B,M\in\ca$
calculates the number of three-cycles of exact sequences as follow:
\[
\begin{tikzpicture}
	\node (1) at (0,0) {$N$};
	\node (2) at (1,0) {$M$};
	\node (3) at (2,0){$L$,};
	\node (6) at (0.5,0.9){$A$};
	\node (5) at (1,1.8){$I$};
	\node (4) at (1.5,.9){$B$};
	\draw[->] (2) --node[above ]{} (1);
	\draw[->] (3) --node[below ]{} (2);
	\draw[->] (4) --node[above]{} (3);
	\draw[->] (5) --node[below ]{} (4);
	\draw[->] (6) --node[below ]{} (5);
	\draw[->] (1) --node[below ]{} (6);
	\end{tikzpicture}
\]where each two arrows on the same line form a short exact sequence.

The associativity for the $\Delta$-Hall numbers relies on Green's formula in $\ca$.
The main results of this paper indicate that the $\Delta$-Hall algebra $\mathcal{H}_{\Delta}(\mathcal{A})$ is isomorphic to the 1-periodic derived Hall algebra $\cd\ch_1(\mathcal{A})$, while the extended $\Delta$-Hall algebra $\widetilde{\mathcal{H}}_{\Delta}(\mathcal{A})$ is isomorphic to the $\imath$Hall algebra $\iH(\mathcal{A})$. Moreover, by twisting on $\widetilde{\mathcal{H}}_{\Delta}(\mathcal{A})$ we recover the semi-derived Hall algebras.
When applied to the the nilpotent representation category $\mathcal{A}={\rm rep^{nil}}(\mathbf{k} Q)$ for an arbitrary quiver $Q$ without loops, we obtain new realizations of the $\imath$quantum group $\Ui_{|v= \sqq}$ and the universal $\imath$quantum group $\tUi_{|v= \sqq}$ associated to $Q$. We summarized the relations between these algebras in the following commutative diagram:

%
\[
\begin{tikzpicture}
\node (1) at (-5,0){$\tUi_{|v= \sqq}$};
\node (2) at (0,0){$\iH(\bfk Q)$};
\node (3) at (5,0) {$\widetilde{\mathcal{H}}_{\Delta}(\bfk Q)$};
\node (4) at (-5,-3){$\Ui_{|v= \sqq}$};
\node (5) at (0,-3){$\cd\ch_1(\bfk Q)$};
\node (6) at (5,-3) {$\mathcal{H}_{\Delta}(\bfk Q)$};
\draw[>->] (1) --node[above]{} (2);
\draw[->] (2) --node[above ]{$\cong$} (3);
\draw[>->] (4) --node[above]{} (5);
\draw[->] (5) --node[above ]{$\cong$} (6);
\draw[<<-] (4) --node[below right]{\cite[Prop 6.2]{LW19a}} (1);
\draw[<<-] (5) --node[below right]{\cite[Thm 3.4]{CLR}} (2);
\draw[<<-] (6) --node[left]{} (3);
\node (3) at (-2.5,-0.5) {\cite[Thm 9.6]{LW20a}};
\node (3) at (2.5,-0.4) {Prop \ref{prop4.8}};
\node (6) at (2.5,-3.4) {Prop \ref{prop3.2}};
\node (3) at (-2.7,-3.5) {\cite[Thm 4.4]{CLR}};
\node (3) at (6.1,-1.8) {Prop \ref{cor4.6}\quad };
\end{tikzpicture}
\]

The paper is organized as follows. In Section 2, we prove the associativity for the $\Delta$-Hall numbers and define the $\Delta$-Hall algebras.
We show that $\Delta$-Hall algebras are isomorphic to derived Hall algebras in Section 3, which provide a new realization for $\imath$quantum groups.
Section 4 is devoted to formulate the isomorphism between extended $\Delta$-Hall algebras and $\imath$Hall algebras, hence provide a new way to realize the universal $\imath$quantum groups.
In Section 5, by considering the twisting on $\Delta$-Hall algebras we recover the semi-derived Hall algebras.

\medskip
For the convenience of readers, we list different algebras (with multiplication and basis), associated to an essentially small hereditary abelian category $\ca$, involved in this paper as follows:
\\
$\triangleright$ $(\mathcal{H}(\mathcal{A}),\;\diamond)$ $-$ Ringel-Hall algebra of $\ca$ with the basis $\big\{[M ]\mid [M]\in{\rm Iso}(\mathcal{A})\big\}$;
\\
$\triangleright$ $(\mathcal{H}_{\Delta}(\mathcal{A}),\;*)$ $-$ $\Delta$-Hall algebra of $\ca$ with the basis $\big\{[M ]\mid [M]\in{\rm Iso}(\mathcal{A})\big\}$;
\\
$\triangleright$ $(\mathcal{DH}_1(\mathcal{A}),\;*)$ $-$ 1-periodic derived Hall algebra of $\mathcal{A}$ with the basis $\big\{u_{[M]}\mid [M]\in{\rm Iso}(\mathcal{D}_1(\ca))\big\}$;
\\
$\triangleright$ $(\cs\cd\ch(\cc_1(\ca)),\;\diamond)$ $-$ semi-derived Hall algebra of $\cc_1(\ca)$ with the basis $\big\{[M ]\diamond[K_\alpha]\mid [M]\in{\rm Iso}(\mathcal{A}),\alpha\in K_0(\ca)\big\}$;
\\
$\triangleright$ $(\iH(\ca),\;*)$ $-$ $\imath$Hall algebra of $\cc_1(\ca)$ with the basis $\big\{[M ]*[K_\alpha]\mid [M]\in{\rm Iso}(\mathcal{A}),\alpha\in K_0(\ca)\big\}$;
\\
$\triangleright$ $(\widetilde{\mathcal{H}}_{\Delta}(\mathcal{A}),\;*)$ $-$ extended $\Delta$-Hall algebra of $\ca$ with the basis $\big\{[M ][K_\alpha]\mid [M]\in{\rm Iso}(\mathcal{A}),\alpha\in K_0(\ca)\big\}$;
\\
$\triangleright$ $(\widetilde{\widetilde{\mathcal{H}}}_{\Delta}(\mathcal{A}),\;*)$ $-$ another extended version of $\mathcal{H}_{\Delta}(\mathcal{A})$ with the basis $\big\{[M ][K_\alpha]\mid [M]\in{\rm Iso}(\mathcal{A}),\alpha\in \dfrac{1}{2}K_0(\ca)\big\}$;
\\
$\triangleright$ $({}_{\varphi}\widetilde{\mathcal{H}}_{\Delta}(\mathcal{A}),\;\divideontimes)$ $-$ twisted extended $\Delta$-Hall algebra of $\ca$ with the basis $\big\{[M ][K_\alpha]\mid [M]\in{\rm Iso}(\mathcal{A}),\alpha\in K_0(\ca)\big\}$.

\section{Definition of $\Delta$-Hall algebras}

In this paper, we take the field $\mathbf{k}=\mathbb F_q$, a finite field of $q$ elements.
Let $\mathcal{A}$ be an essentially small hereditary abelian category, linear over $\mathbf{k}$. Assume $\ca$ has finite morphism and extension spaces, i.e.,
\[
|\Hom_\ca(M,N)|<\infty,\quad |\Ext^1_\ca(M,N)|<\infty,\,\,\forall M,N\in\ca.
\]

Inspired by the constructions of the Ringel-Hall algebras due to Ringel \cite{R1} and the $\imath$Hall algebras due to \cite{LW20a,LP16}, we will define a new kind of algebras associated to $\ca$ in this section, called $\Delta$-Hall algebras.

\subsection{Ringel-Hall algebras and Green's formula}
First we will recall the definition of Ringel-Hall algebra of $\ca$.
Let ${\rm Iso}(\mathcal{A})$ be the collection of isomorphic classes in $\mathcal{A}$.
For any $M\in\mathcal{A}$, we always use the notation $a_M$ to denote the cardinality of the automorphism group ${\rm Aut}(M)$ of $M$.

For any three objects $A,B,M\in\ca$,
define $\Ext_\ca^1(A,B)_M\subseteq \Ext_\ca^1(A,B)$ as the subset parameterizing extensions whose middle term is isomorphic to $M$.
The {\em Ringel-Hall algebra} (or {\em Hall algebra} for short) {$\mathcal{H}(\mathcal{A})$} of
$\mathcal{A}$ is defined to be the $\mathbb{Q}$-vector space with the basis
$\big\{[M ]\mid [M]\in{\rm Iso}(\mathcal{A})\big\}$,
endowed with the multiplication defined by (see \cite{Br13})
\[
{[A]\diamond[B]=\sum_{[M]\in{\rm Iso}(\mathcal{A})}\frac{|\Ext_\ca^1(A,B)_M|}{|\Hom_\ca(A,B)|}\cdot [M].}
\]
We remark that the Ringel-Hall algebra used here is the dual version of the original one defined in \cite{R1}.
Throughout the paper, when we sum for $[M]$, we take all the elements $[M]\in{\rm Iso}(\mathcal{A})$ unless stated otherwise.

For any three objects $A,B,M\in\ca$, the \emph{Hall number} $F^M_{AB}$ is given by
\begin{align}
 \label{eq:Fxyz}
F_{AB}^M:= \big |\{X\subseteq M \mid X \cong B,  M/X\cong A\} \big |.
\end{align}
The Riedtmann-Peng formula states that
 \[F^M_{AB}=\frac{|\Ext_\ca^1(A,B)_M|}{|\Hom_\ca(A,B)|} \cdot \frac{a_M}{a_Aa_B}.
\]

The associativity of Hall numbers tells that
\begin{align*}\sum\limits_{[X]}F_{AB}^XF_{XC}^M=\sum\limits_{[Y]}F_{AY}^MF_{BC}^Y,
\end{align*} which will be denoted by $F_{ABC}^M$ for short.

Since $\ca$ is hereditary, the following famous \emph{Green's formula} \cite[Theorem 2]{G} holds:
\begin{align}\label{Green's formula}
&\sum_{[E]} F_{MN}^{E}F_{XY}^{E}\frac{1}{a_E}
=\sum\limits_{[A],[B],[C],[D]}q^{-\langle A,D\rangle}F_{AB}^{M}F_{CD}^{N}F_{AC}^{X}
F_{BD}^{Y}\frac{a_Aa_Ba_Ca_D}{a_Ma_Na_Xa_Y},
\end{align}
where $\langle A,D\rangle={\rm dim\,Hom}_{\ca}(A,D)-{\rm dim\,Ext}^1_{\ca}(A,D)$ is the Euler form in $\ca$.

\subsection{$\Delta$-Hall algebras}
In order to give a realization of $\imath$quantum groups, Lu-Wang \cite{LW20a} introduce the notion of $\imath$Hall algebras. More precisely,
they firstly considered the Ringel-Hall algebra $\mathcal{H}(\mathcal{C}_1(\ca))$ of the category $\mathcal{C}_1(\ca)$ of 1-periodic complexes over $\ca$, and then used the techniques of quotient and localization successively to obtain the so-called semi-derived Hall algebra $\mathcal{SDH}(\mathcal{C}_1(\ca))$ of $\mathcal{C}_1(\ca)$, and finally by using certain twisting to obtain the $\imath$Hall algebra $\iH(\ca)$. It turns out that the central reduction of the $\imath$Hall algebra provides a realization of the $\imath$quantum group.

Observe that the construction of $\imath$Hall algebras $\iH(\ca)$ is complicated, and it is difficult to obtain the basis of $\iH(\ca)$. In this paper, inspired by Lu-Wang's multiplication formula, we hope to define a new kind of Hall algebra $\mathcal{H}_{\Delta}(\ca)$ following Ringel's framework, called the {\em $\Delta$-Hall algebra}, in order to provide a new way to realize the $\imath$quantum groups.

\medskip


Let $\sqq=\sqrt{q}$.
For any three objects $A,B,M\in \ca$, we define the \emph{$\Delta$-Hall number} as follows,
 \begin{equation}\label{F hat}\widehat{F}^M_{AB}=\sum_{[L],[I],[N]} \sqq^{\langle L,I,N\rangle}\cdot \frac{a_L a_I a_N }{a_M}\cdot F_{LI}^BF_{NL}^MF_{IN}^A,\end{equation}
where
\begin{equation}\label{short hand notation for LIN}\langle L,I,N\rangle=\langle L,I\rangle+\langle I,I\rangle+\langle I,N\rangle-\langle L,N\rangle
.\end{equation}
Roughly speaking, the $\Delta$-Hall number $\widehat{F}_{AB}^M$ calculates the number of three-cycles of the form

 \[
\begin{tikzpicture}
    	\node (1) at (0,0) {$N$};
	\node (2) at (1,0) {$M$};
	\node (3) at (2,0){$L$.};
	\node (6) at (0.5,0.9){$A$};
	\node (5) at (1,1.8){$I$};
	\node (4) at (1.5,.9){$B$};
	\draw[->] (2) --node[above ]{} (1);
	\draw[->] (3) --node[below ]{} (2);
	\draw[->] (4) --node[above]{} (3);
	\draw[->] (5) --node[below ]{} (4);
	\draw[->] (6) --node[below ]{} (5);
	\draw[->] (1) --node[below ]{} (6);
	\end{tikzpicture}
\]
Here, each two adjacent arrows on the same line form a short exact sequence in $\ca$.

We have the following associativity for the $\Delta$-Hall numbers, where Green's formula \eqref{Green's formula} plays key roles.

\begin{proposition}
\label{lem2.4}
For any objects $A,B,C,M\in\mathcal{A}$, the following equation holds:
\begin{equation}\label{associativity}
\sum_{[X]}\widehat{F}_{AB}^X\widehat{F}_{XC}^M=\sum_{[Y]}\widehat{F}_{AY}^M\widehat{F}_{BC}^Y.
\end{equation}
\end{proposition}

\begin{proof}
 Consider the left side of \eqref{associativity}, which calculates the following three-cycles of exact sequences:
\[
\begin{tikzpicture}
	\node (1) at (0,0) {$N_1$};
	\node (2) at (1,0) {$X$};
	\node (3) at (2,0){$L_1$};
	\node (6) at (0.5,0.9){$A$};
	\node (5) at (1,1.8){$I_1$};
	\node (4) at (1.5,.9){$B$};
	\draw[->] (2) --node[above ]{} (1);
	\draw[->] (3) --node[below ]{} (2);
	\draw[->] (4) --node[above]{} (3);
	\draw[->] (5) --node[below ]{} (4);
	\draw[->] (6) --node[below ]{} (5);
	\draw[->] (1) --node[below ]{} (6);
	
		\node (1) at (0+5,0) {$N_2$};
	\node (2) at (1+5,0) {$M$};
	\node (3) at (2+5,0){$L_2$.};
	\node (6) at (0.5+5,0.9){$X$};
	\node (5) at (1+5,1.8){$I_2$};
	\node (4) at (1.5+5,.9){$C$};
	\draw[->] (2) --node[above ]{} (1);
	\draw[->] (3) --node[below ]{} (2);
	\draw[->] (4) --node[above]{} (3);
	\draw[->] (5) --node[below ]{} (4);
	\draw[->] (6) --node[below ]{} (5);
	\draw[->] (1) --node[below ]{} (6);
	\end{tikzpicture}
\]
We have
\begin{align*}
    \sum_{[X]}\widehat{F}_{AB}^X\widehat{F}_{XC}^M&=\sum_{[X]}\sum_{[L_1],[I_1],[N_1]} \sqq^{\langle L_1,I_1,N_1\rangle }\cdot \frac{a_{L_1}a_{I_1} a_{N_1}}{a_X}  \cdot F_{L_1 I_1}^BF_{N_1 L_1}^XF_{I_1 N_1}^A\\&
    \qquad\cdot \sum_{[L_2],[I_2],[N_2]} \sqq^{\langle L_2,I_2,N_2\rangle}
    \cdot \frac{a_{L_2} a_{I_2} a_{N_2}}{a_M}  \cdot F_{L_2 I_2}^CF_{N_2 L_2}^MF_{I_2 N_2}^X\\&=
    \sum_{\substack{[L_1],[I_1],[N_1],\\ [L_2],[I_2],[N_2]}
} \sqq^{\langle L_1,I_1,N_1\rangle+ \langle L_2,I_2,N_2\rangle}
    \cdot \frac{a_{L_1} a_{I_1} a_{N_1} a_{L_2} a_{I_2} a_{N_2}}{a_M}  \\&
    \qquad\cdot  F_{L_1 I_1}^BF_{I_1 N_1}^AF_{L_2 I_2}^CF_{N_2 L_2}^M
    \sum_{[X]}F_{N_1 L_1}^XF_{I_2 N_2}^X\frac{1}{a_X}.
    \end{align*}
Using Green's formula, we obtain

    \begin{align*}
    \sum_{[X]}F_{N_1 L_1}^XF_{I_2 N_2}^X\frac{1}{a_X}=\sum_{[G_1],[G_2],[G_3],[G_4]}q^{-\langle G_3,G_1\rangle}F_{G_2G_1}^{L_1} F_{G_3 G_2}^{I_2} F_{G_3G_4}^{N_1} F_{G_4 G_1}^{N_2}\frac{a_{G_1}a_{G_2}a_{G_3}a_{G_4}}{a_{L_1}a_{I_2}a_{N_1}a_{N_2}}.
    \end{align*}
In the remaining of the proof, we use the simplified notation $\langle \sum\limits_{i}A_i,\sum\limits_{j}D_j\rangle$ (without confusion) to denote $\sum\limits_{i,j}\langle A_i,D_j\rangle$ for any $A_i,D_j\in\ca$.
Then
    we get
    \begin{align}\notag
        \sum_{[X]}\widehat{F}_{AB}^X\widehat{F}_{XC}^M&=
    \sum_{\substack{[L_1],[I_1],[N_1],\\ [L_2],[I_2],[N_2]}
}
 \sqq^{\langle L_1,I_1,N_1\rangle+ \langle L_2,I_2,N_2\rangle}
    \cdot \frac{a_{L_1} a_{I_1} a_{N_1} a_{L_2} a_{I_2} a_{N_2}}{a_M}  \cdot F_{L_1 I_1}^BF_{I_1 N_1}^AF_{L_2 I_2}^CF_{N_2 L_2}^M \\\notag&\qquad\cdot
    \sum_{[G_1],[G_2],[G_3],[G_4]}q^{-\langle G_3,G_1\rangle}F_{G_2G_1}^{L_1} F_{G_3 G_2}^{I_2} F_{G_3G_4}^{N_1} F_{G_4 G_1}^{N_2}\frac{a_{G_1}a_{G_2}a_{G_3}a_{G_4}}{a_{L_1}a_{I_2}a_{N_1}a_{N_2}}\\\notag&=
    \sum_{\substack{[L_2],[I_1],\\ [G_1],[G_2],[G_3],[G_4]}
} \sqq^{\langle G_1+G_2,I_1,G_3+G_4\rangle+ \langle L_2,G_2+G_3,G_4+G_1\rangle-2\langle G_3,G_1\rangle}\cdot\frac{a_{I_1} a_{L_2} a_{G_1}a_{G_2}a_{G_3}a_{G_4}}{a_M}
     \\\notag& \qquad\cdot \sum_{[L_1]}F_{L_1 I_1}^BF_{G_2G_1}^{L_1} \sum_{[N_1]}F_{I_1 N_1}^AF_{G_3G_4}^{N_1}
    \sum_{[I_2]}F_{L_2 I_2}^CF_{G_3 G_2}^{I_2}
    \sum_{[N_2]}F_{N_2 L_2}^MF_{G_4 G_1}^{N_2}\\\notag&=
     \sum_{\substack{[L_2],[I_1],\\\notag [G_1],[G_2],[G_3],[G_4]}
} \sqq^{\langle G_1+G_2,I_1,G_3+G_4\rangle+ \langle L_2,G_2+G_3,G_4+G_1\rangle-2\langle G_3,G_1\rangle}
     \\\label{left hand side}& \qquad\cdot\frac{a_{I_1} a_{L_2} a_{G_1}a_{G_2}a_{G_3}a_{G_4}}{a_M}\cdot F_{G_2G_1I_1}^{B} F_{I_1G_3G_4}^{A}
    F_{L_2G_3 G_2}^{C}
F_{G_4 G_1L_2}^{M}.    \end{align}

    On the other hand, consider the right-hand side of \eqref{associativity}, which calculates the following three-cycles of exact sequences:  \[
    \begin{tikzpicture}
	\node (1) at (0,0) {$N_3$};
	\node (2) at (1,0) {$M$};
	\node (3) at (2,0){$L_3$};
	\node (6) at (0.5,0.9){$A$};
	\node (5) at (1,1.8){$I_3$};
	\node (4) at (1.5,.9){$Y$};
	\draw[->] (2) --node[above ]{} (1);
	\draw[->] (3) --node[below ]{} (2);
	\draw[->] (4) --node[above]{} (3);
	\draw[->] (5) --node[below ]{} (4);
	\draw[->] (6) --node[below ]{} (5);
	\draw[->] (1) --node[below ]{} (6);
	
		\node (1) at (0+5,0) {$N_4$};
	\node (2) at (1+5,0) {$Y$};
	\node (3) at (2+5,0){$L_4$.};
	\node (6) at (0.5+5,0.9){$B$};
	\node (5) at (1+5,1.8){$I_4$};
	\node (4) at (1.5+5,.9){$C$};
	\draw[->] (2) --node[above ]{} (1);
	\draw[->] (3) --node[below ]{} (2);
	\draw[->] (4) --node[above]{} (3);
	\draw[->] (5) --node[below ]{} (4);
	\draw[->] (6) --node[below ]{} (5);
	\draw[->] (1) --node[below ]{} (6);
	\end{tikzpicture}
    \]
We have
    \begin{align*}
    \sum_{[Y]}\widehat{F}_{AY}^M\widehat{F}_{BC}^Y&=\sum_{[Y]}\sum_{[L_3],[I_3],[N_3]} \sqq^{\langle L_3,I_3,N_3\rangle }\cdot \frac{a_{L_3}a_{I_3} a_{N_3}}{a_M}  \cdot F_{L_3 I_3}^YF_{N_3 L_3}^MF_{I_3 N_3}^A\\&\qquad
    \cdot \sum_{[L_4],[I_4],[N_4]} \sqq^{\langle L_4,I_4,N_4\rangle}
    \cdot \frac{a_{L_4} a_{I_4} a_{N_4}}{a_Y}  \cdot F_{L_4 I_4}^CF_{N_4 L_4}^YF_{I_4 N_4}^B\\&=
    \sum_{\substack{[L_3],[I_3],[N_3],\\ [L_4],[I_4],[N_4]}
} \sqq^{\langle L_3,I_3,N_3\rangle+ \langle L_4,I_4,N_4\rangle}
    \cdot \frac{a_{L_3} a_{I_3} a_{N_3} a_{L_4} a_{I_4} a_{N_4}}{a_M}  \\&\qquad\cdot  F_{I_3 N_3}^AF_{N_3 L_3}^MF_{L_4 I_4}^CF_{I_4 N_4}^B
    \sum_{[Y]}F_{L_3 I_3}^YF_{N_4 L_4}^Y\frac{1}{a_Y}.
    \end{align*}
By using Green's formula again, we get
    \begin{align*}
        \sum_{[Y]}\widehat{F}_{AY}^M\widehat{F}_{BC}^Y&=
    \sum_{\substack{[L_3],[I_3],[N_3],\\ [L_4],[I_4],[N_4]}
}
 \sqq^{\langle L_3,I_3,N_3\rangle+ \langle L_4,I_4,N_4\rangle}
    \cdot \frac{a_{L_3} a_{I_3} a_{N_3} a_{L_4} a_{I_4} a_{N_4}}{a_M} \cdot F_{I_3 N_3}^AF_{N_3 L_3}^MF_{L_4 I_4}^CF_{I_4 N_4}^B \\&\qquad\cdot
    \sum_{[H_1],[H_2],[H_3],[H_4]}q^{-\langle H_3,H_1\rangle}F_{H_2H_1}^{L_4} F_{H_3 H_2}^{L_3} F_{H_3H_4}^{N_4} F_{H_4 H_1}^{I_3}\frac{a_{H_1}a_{H_2}a_{H_3}a_{H_4}}{a_{L_4}a_{I_3}a_{N_4}a_{L_3}}\\&=
    \sum_{\substack{[N_3],[I_4],\\ [H_1],[H_2],[H_3],[H_4]}
} \sqq^{\langle H_2+H_3,H_4+H_1,N_3\rangle+ \langle H_1+H_2,I_4,H_3+H_4\rangle-2\langle H_3,H_1\rangle}
    \cdot \frac{a_{N_3}a_{I_4}a_{H_1} a_{H_2}a_{H_3}a_{H_4}}{a_M}  \\&\qquad\cdot \sum_{[I_3]}F_{I_3N_3}^AF_{H_4H_1}^{I_3}  \sum_{[L_3]}F_{N_3L_3}^MF_{H_3H_2}^{L_3}
    \sum_{[L_4]}F_{L_4 I_4}^CF_{H_2 H_1}^{L_4}
    \sum_{[N_4]}F_{I_4N_4}^BF_{H_3 H_4}^{N_4} \\&=
    \sum_{\substack{[N_3],[I_4],\\ [H_1],[H_2],[H_3],[H_4]}
} \sqq^{\langle H_2+H_3,H_4+H_1,N_3\rangle+ \langle H_1+H_2,I_4,H_3+H_4\rangle-2\langle H_3,H_1\rangle}       \\
    &\qquad{\cdot\frac{a_{N_3}a_{I_4}a_{H_1} a_{H_2}a_{H_3}a_{H_4}}{a_M}\cdot F^A_{H_4H_1N_3}  F_{N_3H_3H_2}^M    F_{H_2 H_1 I_4}^C     F_{I_4H_3 H_4}^{B}. }   \end{align*}
Observe that the last summation above is taken over $[N_3], [I_4],[H_1], [H_2],[H_4],[H_3]$ for arbitrary isomorphism classes from $\ca$. By replacing them by $[G_4], [G_2],[G_3],[L_2],[I_1],[G_1]$ respectively, we obtain

 \begin{align}\notag
        \sum_{[Y]}\widehat{F}_{AY}^M\widehat{F}_{BC}^Y&=
    \sum_{\substack{[G_4],[G_2],\\ [G_3],[L_2],[I_1],[G_1]}
} \sqq^{\langle L_2+G_1,I_1+G_3,G_4\rangle+ \langle G_3+L_2,G_2,G_1+I_1\rangle-2\langle G_1,G_3\rangle}\\\label{right hand side}&
    \qquad\cdot \frac{a_{G_4} a_{G_2}a_{G_3}a_{L_2}a_{G_1}a_{I_1}}{a_M}
    \cdot F_{I_1G_3G_4}^A
    F_{G_4G_1L_2}^{M}
    F_{L_2G_3 G_2}^CF_{G_2G_1I_1}^B.
    \end{align}

    { Comparing \eqref{left hand side} and \eqref{right hand side}, in order to prove \eqref{associativity}, it remains to show the following relation:}
    \[
   \langle G_1+G_2,I_1,G_3+G_4\rangle+ \langle L_2,G_2+G_3,G_4+G_1\rangle-2\langle G_3,G_1\rangle\]\[=\langle L_2+G_1,I_1+G_3,G_4\rangle+ \langle G_3+L_2,G_2,G_1+I_1\rangle-2\langle G_1,G_3\rangle,
    \]
{which follows easily but tediously, via expanding both sides of the equation by using \eqref{short hand notation for LIN}.}
\end{proof}

\begin{theorem}
\label{thm2.4}
The $\mathbb{Q}(\sqq)$-vector space $\mathcal{H}_{\Delta}(\mathcal{A})$ with the basis
$
\big\{[M ]\mid [M]\in{\rm Iso}(\mathcal{A})\big\},
$
endowed with the multiplication defined by\[
[A]*[B]=\sum_{[M]}\widehat{F}_{AB}^M\cdot [M],
\] forms an associative algebra with the unit $[0]$, {called the
\emph{$\Delta$-Hall algebra} of $\ca$.}
\end{theorem}

\begin{proof}
We have\begin{align*}
    &([A]*[B])*[C]\\&=
    \sum_{[X]}\widehat{F}_{AB}^X\cdot [X]*[C]\\&=
    \sum_{[X]}\widehat{F}_{AB}^X\cdot\sum_{[M]}\widehat{F}_{XC}^M \cdot[M]\\&=
    \sum_{[M]}\sum_{[X]}\widehat{F}_{AB}^X\widehat{F}_{XC}^M \cdot[M].
\end{align*}

    On the other hand, we have
\begin{align*}
    &[A]*([B]*[C])\\&=
    [A]*(\sum_{[Y]}\widehat{F}_{BC}^Y\cdot [Y])\\&=
    \sum_{[Y]}\widehat{F}_{BC}^Y\cdot [A]*[Y]\\&=
    \sum_{[M]}\sum_{[Y]}\widehat{F}_{AY}^M\widehat{F}_{BC}^Y \cdot[M].
    \end{align*}
    By Proposition \ref{lem2.4}, we immediately get\[
    ([A]*[B])*[C]=[A]*([B]*[C]).
    \]
Clearly, $[0]$ is the unit by definition. Then we are done.
\end{proof}

\section{$\Delta$-Hall algebra and $\imath$quantum group}

In this section, we establish an isomorphism between $\Delta$-Hall algebras and 1-periodic derived Hall algebras, and then provide a new realization of $\imath$quantum groups.

\subsection{$\Delta$-Hall algebras and 1-periodic derived Hall algebras}

Recall that $\mathcal{A}$ is an essentially small hereditary abelian category over the finite field $\mathbf{k}=\mathbb F_q$. Let $\mathcal{D}_1(\mathcal{A})$ be the derived category of 1-periodic complexes on $\mathcal{A}$.
For any object $M^\bullet=(M,d)\in\mathcal{D}_1(\mathcal{A})$, there exists an isomorphism $(M,d)\cong (H(M^\bullet),0)$ in $\mathcal{D}_1(\mathcal{A})$, where $H(M^\bullet)=\Ker(d)/\Im(d)\in\ca$ is the cohomology group of $M^\bullet$. Hence, the isoclasses $\rm{Iso}(\mathcal{D}_1(\mathcal{A}))$ of $\mathcal{D}_1(\mathcal{A})$ coincide with the isoclasses $\rm{Iso}(\mathcal{A})$ of $\mathcal{A}$. From now on, we identify $\rm{Iso}(\mathcal{D}_1(\mathcal{A}))$ and $\rm{Iso}(\mathcal{A})$.

For any object $M\in \mathcal{A}$, we denote $|{\rm Aut}_{\mathcal{D}_1(\mathcal{A})}(M)|$ by $\tilde{a}_{M}$. Then we have $$\tilde{a}_M=a_M\cdot |\Ext^1_\mathcal{A}(M,M)|.$$
For any objects $A,B,M\in\mathcal{A}$, following \cite{XC13} we use the notations
\[
(A,B)_M:=\{f\in\Hom_{\mathcal{D}_1(\mathcal{A})}(A,B)|{\rm Cone}(f)\cong M\}
\]
and
\[
\{A,B\}:=\frac{1}{|\Hom_{\mathcal{D}_1(\mathcal{A})}(A,B)|}.
\]
Then by \cite[Corollary 2.7]{XC13} we have
\[
\frac{|(B,M)_A|}{\tilde{a}_B}\sqrt{\frac{\{B,M\}}{\{B,B\}}}=\frac{|(M,A)_B|}{\tilde{a}_A}\sqrt{\frac{\{M,A\}}{\{A,A\}}}.
\]
We denote this number by $G_{AB}^{M}$ in this paper, and call it the \emph{derived Hall number}, which satisfies the following
derived Riedtmann-Peng formula by \cite[Proposition 3.3]{SCX18}:

\begin{align*}\label{derive RP formula}
G_{AB}^M=\frac{\tilde{a}_M\cdot|(A,B)_M|}{\tilde{a}_A\cdot\tilde{a}_B}\sqrt{\frac{\{A,B\}\{M,M\}}{\{A,A\}\{B,B\}}}.
\end{align*}

The \emph{1-periodic derived Hall algebra} $\mathcal{DH}_1(\mathcal{A})$ of $\ca$ is a $\mathbb{Q}(\sqq)$-vector space with the basis $\big\{u_{[M]}\,|\, [M]\in \rm{Iso}(\mathcal{D}_1(\mathcal{A})):=\rm{Iso}(\mathcal{A})\big\}$, endowed with the multiplication defined by\[
u_{[A]}*u_{[B]}=\sum_{[M]}G_{AB}^M\cdot u_{[M]}.
\]


The following proposition indicates that the derived Hall number $G_{AB}^M$ can be calculated inside the abelian category $\ca$.

\begin{proposition}
\cite[Proposition 3.5]{CLR}
For any objects $A,B,M\in\mathcal{A}$, we have
\begin{equation}\label{G_AB^M}
G_{AB}^M=\sum_{[L],[I],[N]}\sqq^{\langle L,I\rangle+\langle I,I\rangle+\langle I,N\rangle-\langle L,N\rangle}\cdot \frac{a_L a_I a_N}{a_A a_B}\cdot F^B_{LI}F_{NL}^MF^A_{IN}.
\end{equation}
\end{proposition}

Comparing \eqref{G_AB^M} with \eqref{F hat}, and using \eqref{short hand notation for LIN},
we can obtain an isomorphism from the 1-periodic derived Hall algebra $\mathcal{DH}_1(\mathcal{A})$ to the $\Delta$-Hall algebra $\mathcal{H}_{\Delta}(\mathcal{A})$.

\begin{proposition}
\label{prop3.2}
There is an algebra isomorphism \begin{align*}
\Xi_\ca:\mathcal{DH}_1(\mathcal{A})
&\longrightarrow \mathcal{H}_{\Delta}(\mathcal{A}).\\
u_{[M]}&\mapsto \frac{1}{a_M}\cdot [M]
\end{align*}
\end{proposition}

\begin{proof}
Recall that we have the identity $\rm{Iso}(\mathcal{D}_1(\mathcal{A}))=\rm{Iso}(\mathcal{A})$. For any $A,B\in\mathcal{A}$, we have
\begin{align*}
\Xi_\ca(u_{[A]}*u_{[B]})&=\Xi_\ca(\sum_{[M]}G_{AB}^M\cdot u_{[M]})\\&=
\sum_{[M]}\ \sum_{[L],[I],[N]}\sqq^{\langle L,I,N\rangle}\cdot \frac{a_L a_I a_N}{a_A a_B}\cdot F^B_{LI}F_{NL}^MF^A_{IN}\cdot \Xi_\ca(u_{[M]})\\&=
\sum_{[M]}\ \sum_{[L],[I],[N]}\sqq^{\langle L,I,N\rangle}\cdot \frac{a_L a_I a_N}{a_A a_B a_M}\cdot F^B_{LI}F_{NL}^MF^A_{IN}\cdot [M].
\end{align*}
On the other hand,
\begin{align*}
    \Xi_\ca(u_{[A]})*\Xi_\ca(u_{[B]})&=\frac{1}{a_A a_B} [A]* [B]\\&=\frac{1}{a_A a_B}\cdot \sum_{[M]}\widehat{F}_{AB}^M\cdot [M]\\&=\frac{1}{a_Aa_B}\cdot \sum_{[M]}\ \sum_{[L],[I],[N]}\sqq^{\langle L,I,N\rangle}\cdot \frac{a_L a_I a_N}{a_M}\cdot F^B_{LI}F_{NL}^MF^A_{IN}\cdot [M]\\&=
    \sum_{[M]}\ \sum_{[L],[I],[N]}\sqq^{\langle L,I,N\rangle}\cdot \frac{a_L a_I a_N}{a_A a_B a_M}\cdot F^B_{LI}F_{NL}^MF^A_{IN}\cdot [M].
\end{align*}
Hence $\Xi_\ca(u_{[A]})*\Xi_\ca(u_{[B]})=\Xi_\ca(u_{[A]}*u_{[B]})$ and then $\Xi_\ca$ is an algebra homomorphism. Since the algebras $\mathcal{DH}_1(\mathcal{A})$
and $\mathcal{H}_{\Delta}(\mathcal{A})$ share the same basis $\rm{Iso}(\mathcal{D}_1(\mathcal{A}))=\rm{Iso}(\mathcal{A})$, we obtain that $\Xi_\ca$ is bijective. Therefore, $\Xi_\ca$ is an isomorphism of algebras.
\end{proof}

\subsection{New realization of $\imath$Quantum groups of split type}

Now we will recall the basic setting on (universal) quantum groups and (universal) $\imath$quantum groups. Let $Q$ be a quiver (without loops) with vertex set $Q_0= \I$.
Let $n_{ij}$ be the number of edges connecting vertices $i$ and $j$. The symmetric generalized Cartan matrix of the underlying graph of $Q$ is denoted by $C=(c_{ij})_{i,j \in \I}$, where $c_{ij}=2\delta_{ij}-n_{ij}.$ Let $\fg$ be the corresponding Kac-Moody Lie algebra of $Q$ and $C$.

Let $\bv$ be an indeterminant. For $r,m \in \N$, denote by
\[
[r]=\frac{\bv^r-\bv^{-r}}{\bv-\bv^{-1}},
\quad
[r]!=\prod_{i=1}^r [i], \quad \qbinom{m}{r} =\frac{[m][m-1]\ldots [m-r+1]}{[r]!}.
\]

The \emph{universal quantum group} $\tU := \tU_\bv(\fg)$ of $\fg$ is defined to be the $\Q(\bv)$-algebra generated by $E_i,F_i, \tK_i,\tK_i'$, $i\in \I$, where $\tK_i, \tK_i'$ are invertible, subject to the following relations for $i, j \in \I$:
\begin{align}
	[E_i,F_j]= \delta_{ij} \frac{\tK_i-\tK_i'}{\bv-\bv^{-1}},  &\qquad [\tK_i,\tK_j]=[\tK_i,\tK_j']  =[\tK_i',\tK_j']=0,
	\label{eq:KK}
	\\
	\tK_i E_j=\bv^{c_{ij}} E_j \tK_i, & \qquad \tK_i F_j=\bv^{-c_{ij}} F_j \tK_i,
	\label{eq:EK}
	\\
	\tK_i' E_j=\bv^{-c_{ij}} E_j \tK_i', & \qquad \tK_i' F_j=\bv^{c_{ij}} F_j \tK_i',
	\label{eq:K2}
\end{align}
and the quantum Serre relations for $i\neq j \in \I$,
\begin{align}
	& \sum_{r=0}^{1-c_{ij}} (-1)^r  E_i^{(r)} E_j  E_i^{(1-c_{ij}-r)}=0,
	\label{eq:serre1} \\
	& \sum_{r=0}^{1-c_{ij}} (-1)^r   F_i^{(r)} F_j  F_i^{(1-c_{ij}-r)}=0.
	\label{eq:serre2}
\end{align}
Here, $[A, B]=AB-BA$, and for any  $n\ge 1, i\in \I$, \[
E_i^{(n)} =\frac{E_i^n}{[n]!}, \quad F_i^{(n)} =\frac{F_i^n}{[n]!}.
\]
Note that $\tK_i \tK_i'\; (i\in \I)$ are central in $\tU$.

The \emph{universal $\imath$quantum group of split type} $\widetilde{\bU}^\imath$ 
is defined to be the $\Q(v)$-subalgebra of $\tU$ generated by
\[
B_i= F_i +  E_{i} \tK_i',
\qquad \tk_i = \tK_i \tK_{i}', \quad \forall i \in \I.
\]
Let $\tU^{\imath 0}$ be the $\Q(v)$-subalgebra of $\tUi$ generated by central elements $\tk_i$, for $i\in \I$.



{Analogous as for $\tU$, the \emph{quantum group} $\bU:=\bU_\bv(\fg)$ is defined to be the $\Q(v)$-algebra generated by $E_i,F_i, K_i, K_i^{-1}$, $i\in \I$, subject to the  relations modified from \eqref{eq:KK}--\eqref{eq:serre2} with $\tK_i$ and $\tK_i'$ replaced by $K_i$ and $K_i^{-1}$, respectively.}

Let $\bvs=(\vs_i)_{i \in \I}\in  (\Q(\bv)^\times)^{\I}$.
The \emph{$\imath$quantum group of split type} $\Ui:=\Ui_{\bvs}$ is the $\Q(v)$-subalgebra of $\bU$ generated by
\[
B_i= F_i+\vs_i E_{i}K_i^{-1},
\qquad  \forall i \in \I.
\]

In order to give a realization of the $\imath$quantum groups via $\Delta$-Hall algebra approach, we consider the nilpotent representation category $\mathcal{A}={\rm rep^{nil}}(\mathbf{k} Q)$. Let $S_i\;(i\in \I)$ be the set of simple $\mathbf{k} Q$-modules.
 We denote by $\mathcal{DH}_1(\mathbf{k} Q):=\mathcal{DH}_1(\mathcal{A})$, $\mathcal{H}_{\Delta}(\mathbf{k} Q):=\mathcal{H}_{\Delta}(\mathcal{A})$ and $\Xi_Q:=\Xi_{\mathcal{A}}$.




\begin{theorem}
\label{thm3.4}
Let $Q$ be an arbitrary quiver without loops. Then there exists a $\mathbb{Q}(\sqq)$-algebra embedding
\begin{align*}
		\Theta_Q: \mathbf{U}^\imath_{|v= \sqq} &\longrightarrow \mathcal{H}_{\Delta}(\mathbf{k} Q),
	\end{align*}
	which sends
	\begin{align*}
		B_i \mapsto -\sqq^{-1}\frac{[S_{i}]}{a_{S_{i}}},
		 \qquad\text{ for }i \in \I.
	\end{align*}
In particular, if $Q$ is a Dynkin quiver, the embedding $\Theta_Q$ is precisely an isomorphism.
\end{theorem}

\begin{proof} Recall from \cite[{Theorem 4.4}]{CLR} that there exists a $\mathbb{Q}(\sqq)$-algebra embedding
	\begin{align*}
		\Psi_Q: \mathbf{U}^\imath_{|v= \sqq} &\longrightarrow \mathcal{DH}_1(\mathbf{k} Q),
	\end{align*}
	which sends
	\begin{align*}
		B_i \mapsto -\sqq^{-1}u_{[S_{i}]},
		 \qquad\text{ for }i \in \I.
	\end{align*}

Let $\Theta_Q$ be the following composition $$\Theta_Q=\Xi_Q\circ \Psi_Q:\mathbf{U}^\imath_{|v= \sqq} \longrightarrow \mathcal{DH}_1(\mathbf{k} Q)\longrightarrow \mathcal{H}_{\Delta}(\mathbf{k} Q).$$
Then the result follows from Proposition \ref{prop3.2} immediately.
\end{proof}

\section{Extended $\Delta$-Hall algebra and universal $\imath$quantum group}

In this section, we define the extended $\Delta$-Hall algebra of $\ca$, and prove that it is isomorphic to the $\imath$Hall algebra of $\mathcal{A}$, hence provide a new realization of the universal $\imath$quantum group associated to $\ca$.

\subsection{Extended $\Delta$-Hall algebras}

Let $K_0(\mathcal{A})$ be the Grothendieck group of $\mathcal{A}$. It is well-known that the Ringel-Hall algebra $\mathcal{H}(\mathcal{A})$ has an extended version via tensoring with the group algebra of $K_0(\mathcal{A})$.
Similarly, the $\Delta$-Hall algebras also admit extended versions, which can be used to realize the universal $\imath$quantum groups.

For technical reason we assume $K_0(\ca)$ is a free abelian group. Both the category of finite-dimensional representations of a quiver, and the category of coherent sheaves over a weighted projective curve, satisfy this assumption.

For any $M\in\ca$, the image of $M$ in $K_0(\ca)$ is denoted by $\widehat{M}$.
The symmetric Euler form in $\ca$ (or in $K_0(\ca)$) is determined by
$(M,N)=\langle M,N\rangle+\langle N,M\rangle$ for any $M,N\in\ca$.
We have the following observation.

{\begin{lemma}
\label{lem4.5}
For any three-cycle in $\ca$ as follows:
\begin{equation}\label{three cycle}
\begin{tikzpicture}
	\node (1) at (0,0) {$N$};
	\node (2) at (1,0) {$M$};
	\node (3) at (2,0){$L$,};
	\node (6) at (0.5,0.9){$A$};
	\node (5) at (1,1.8){$I$};
	\node (4) at (1.5,.9){$B$};
	\draw[->] (2) --node[above ]{} (1);
	\draw[->] (3) --node[below ]{} (2);
	\draw[->] (4) --node[above]{} (3);
	\draw[->] (5) --node[below ]{} (4);
	\draw[->] (6) --node[below ]{} (5);
	\draw[->] (1) --node[below ]{} (6);
	\end{tikzpicture}
\end{equation}
we have
\begin{itemize}
  \item[(i)] $2\widehat{L}=\widehat{M}+\widehat{B}-\widehat{A}$;
  \item[(ii)] $2\widehat{I}=\widehat{A}+\widehat{B}-\widehat{M}$;
  \item[(iii)] $2\widehat{N}=\widehat{M}+\widehat{A}-\widehat{B}$.
\end{itemize}
Consequently,
\begin{equation}\label{bilinear form relations}
2(N,L)=\langle M,M\rangle-\langle A,A\rangle-\langle B,B\rangle+(A,B).
\end{equation}
\end{lemma}

\begin{proof}
The three-cycle \eqref{three cycle} indicates that we have the following exact sequence:
$$\xymatrix{0\ar[r]&N\ar[r]&A\ar[r]&B\ar[r]&M\ar[r]&N\ar[r]&0.}$$
Hence, $2\widehat{N}=\widehat{M}+\widehat{A}-\widehat{B}$, which proves (iii). Similarly, we can obtain (i) and (ii).

Observe that
\begin{align*}
4(N, L)=&( 2\widehat{N}, 2\widehat{L})\\
=&(\widehat{M}+\widehat{A}-\widehat{B}, \widehat{M}+\widehat{B}-\widehat{A})\\
=&
(\widehat{M},\widehat{M})-(\widehat{A}-\widehat{B},\widehat{A}-\widehat{B})\\
=&(\widehat{M},\widehat{M}) -(\widehat{A},\widehat{A})
 -( \widehat{B}, \widehat{B})
+2(\widehat{A},\widehat{B}).
\end{align*}
By assumption, $K_0(\ca)$ is a free abelian group. Then \eqref{bilinear form relations} follows immediately.
\end{proof}}

Denote by $\widetilde{\mathcal{H}}_{\Delta}(\mathcal{A})$ the $\mathbb{Q}(\sqq)$-vector space with the basis
$$
\big\{[M] [K_\alpha]\mid [M]\in{\rm Iso}(\mathcal{A}), {\alpha\in K_0(\mathcal{A})}\big\}.
$$

\begin{proposition}\label{extended version}
The $\mathbb{Q}(\sqq)$-vector space $\widetilde{\mathcal{H}}_{\Delta}(\mathcal{A})$ endowed with the multiplication defined by
\[[A] [K_\alpha]*[B][K_\beta]=\sum_{[M]}\widehat{F}_{AB}^M\cdot[M]
[K_{\frac{\widehat{A}+\widehat{B}-\widehat{M}}{2}+\alpha+\beta}],\]
forms an associative algebra with unit $[0][K_0]$, called the
\emph{extended $\Delta$-Hall algebra} of $\ca$.
\end{proposition}


\begin{proof}
Recall that $$\widehat{F}^M_{AB}=\sum\limits_{[L],[I],[N]} \sqq^{\langle L,I,N\rangle}\cdot \dfrac{a_L a_I a_N }{a_M}\cdot F_{LI}^BF_{NL}^MF_{IN}^A.$$ If $\widehat{F}^M_{AB}\neq 0$, then all of $F_{LI}^B, F_{NL}^M, F_{IN}^A$ are non-zero, which follows that $\widehat{A}+\widehat{B}-\widehat{M}=2\widehat{I}$ by Lemma \ref{lem4.5}. Hence $$\frac{\widehat{A}+\widehat{B}-\widehat{M}}{2}\in K_0(\ca).$$ So the multiplication is well-defined.

For any $[A], [B], [C]\in {\rm Iso}(\mathcal{A})$ and $\alpha, \beta, \gamma \in K_0(\mathcal{A})$, we have
\begin{align*}
    &([A] [K_\alpha]*[B] [K_\beta])*[C] [K_\gamma]\\&=
    \sum_{[X]}\widehat{F}_{AB}^X\cdot [X] [K_{\frac{\widehat{A}+\widehat{B}-\widehat{X}}{2}+\alpha+\beta}]*[C] [K_\gamma]\\&=
    \sum_{[X]}\widehat{F}_{AB}^X\cdot\sum_{[M]}\widehat{F}_{XC}^M \cdot[M] [K_{\frac{\widehat{X}+\widehat{C}-\widehat{M}}{2}+\frac{\widehat{A}+\widehat{B}-\widehat{X}}{2}+\alpha+\beta+\gamma }]\\&=
    \sum_{[M]}\sum_{[X]}\widehat{F}_{AB}^X\widehat{F}_{XC}^M \cdot[M] [K_{\frac{\widehat{A}+\widehat{B}+\widehat{C}-\widehat{M}}{2}+\alpha+\beta+\gamma }],
\end{align*}
and
\begin{align*}
    &[A] [K_\alpha]*([B] [K_\beta]*[C] [K_\gamma])\\&=
    [A] [K_\alpha]*(\sum_{[Y]}\widehat{F}_{BC}^Y\cdot [Y] [K_{\frac{\widehat{B}+\widehat{C}-\widehat{Y}}{2}+\beta+\gamma}])\\&=
    \sum_{[Y]}\widehat{F}_{BC}^Y\cdot [A] [K_\alpha]*[Y] [K_{\frac{\widehat{B}+\widehat{C}-\widehat{Y}}{2}+\beta+\gamma}]\\&=
    \sum_{[M]}\sum_{[Y]}\widehat{F}_{AY}^M\widehat{F}_{BC}^Y \cdot[M] [K_{\frac{\widehat{A}+\widehat{B}+\widehat{C}-\widehat{M}}{2}+\alpha+\beta+\gamma }].
    \end{align*}
    Now it follows from Lemma \ref{lem2.4} that
    \[
    ([A] [K_\alpha]*[B] [K_\beta])*[C] [K_\gamma]=[A] [K_\alpha]*([B] [K_\beta]*[C] [K_\gamma]).
    \]
    Then we are done.
\end{proof}

\subsection{Extended $\Delta$-Hall algebras and $\imath$Hall algebras}
In this subsection we show that the extend $\Delta$-Hall algebra $\widetilde{\mathcal{H}}_{\Delta}(\mathcal{A})$ is isomorphic to the $\imath$Hall algebra of $\ca$.
First let us recall the construction of the $\imath$Hall algebra due to \cite{LP16, LW19a, LR21}.

Denote by $\mathcal{H}(\cc_1(\ca))$ the Ringel-Hall algebra for the category $\cc_1(\ca)$ of $1$-periodic complexes over $\ca$.
Consider the ideal $\mathcal{I}$ of $\mathcal{H}(\mathcal{C}_1(\mathcal{A}))$ generated by
\begin{align*}
&\Big\{ [M^\bullet]-[N^\bullet]\mid H(M^\bullet)\cong H(N^\bullet), \quad \widehat{\Im d_{M^\bullet}}=\widehat{\Im d_{N^\bullet}} \Big\}.
\end{align*}
Denote by
\[
\cs:=\big\{ a[K^\bullet] \in \ch(\cc_1(\ca))/\cI \mid a\in \Q(\sqq)^\times, K^\bullet\in \cc_1(\ca) \text{ acyclic}\big\},
\]
a multiplicatively closed subset of $\ch(\cc_1(\ca))/ \cI$ with the identity $[0]$, which is right Ore and right reversible.
Hence there exists the right localization of
$\ch(\cc_1(\ca))/\cI$ with respect to $\cs$, called the {\em semi-derived Hall algebra} of $\cc_1(\ca)$ in the sense of \cite{LP16, LW19a}, and will be denoted by $\cs\cd\ch(\cc_1(\ca))$.

The {\em $\imath$Hall algebra} $\iH(\ca)$ of $\ca$ is defined to be the twisted semi-derived Hall algebra $\cs\cd\ch(\cc_1(\ca))$ via the restriction functor $\res: \cc_1(\ca)\rightarrow\ca$ {by forgetting differentials}. That is, $\iH(\ca)$ is the $\Q(\sqq)$-algebra on the same vector space as $\cs\cd\ch(\cc_1(\ca))$, equipped with the following modified multiplication
\begin{align}
   \label{eq:tH}
[M^\bullet]* [N^\bullet] =\sqq^{\langle \res(M^\bullet),\res(N^\bullet)\rangle} [M^\bullet]\diamond[N^\bullet];\quad \forall M^\bullet, N^\bullet\in\cc_1(\ca).
\end{align}


By Theorem 4.5 and Proposition 4.10 in \cite{LW19a}, the $\imath$Hall algebra $\iH(\ca)$ has a basis
$$
\big\{[M]*[K_\alpha]\mid [M]\in{\rm Iso}(\mathcal{A}), {\alpha\in K_0(\mathcal{A})}\big\},
$$
and the elements $[K_\alpha]\;(\alpha\in K_0(\mathcal{A}))$ are central in $\iH(\ca)$. {In order to give an explicit description of the multiplication formula for the basis elements in $\iH(\ca)$, we introduce the following notation
for any $[A], [B]\in {\rm Iso}(\mathcal{A})$:
\begin{align}\label{F tilde}\widetilde{F}_{AB}^M=\sum\limits_{[L],[I],[N]} \sqq^{-\langle A,B\rangle} q^{\langle N,L\rangle}\cdot \frac{a_L a_I a_N }{a_M}\cdot F_{LI}^BF_{NL}^MF_{IN}^A.
\end{align}}


{\begin{lemma}\label{tilde and hat} For any $A,B, M\in\ca$,
$$\widetilde{F}_{AB}^M=\sqq^{\langle M,M\rangle-\langle A,A\rangle-\langle B,B\rangle}\cdot\widehat{F}_{AB}^M.$$
\end{lemma}

\begin{proof}
Recall from \eqref{F hat} that
\begin{equation}\label{F hat 2}\widehat{F}^M_{AB}=\sum_{[L],[I],[N]} \sqq^{\langle L,I,N\rangle}\cdot \frac{a_L a_I a_N }{a_M}\cdot F_{LI}^BF_{NL}^MF_{IN}^A.\end{equation}
For $F_{LI}^BF_{NL}^MF_{IN}^A\neq0$, we have $\widehat{A}=\widehat{I}+\widehat{N}$ and $\widehat{B}=\widehat{L}+\widehat{I}$, which implies $$\langle B,A\rangle=\langle L,I\rangle+\langle I,I \rangle+\langle I,N\rangle+\langle L,N\rangle.$$ It follows from \eqref{short hand notation for LIN} that $$\langle B,A\rangle-2\langle L,N\rangle=\langle L,I,N\rangle.$$
Now according to \eqref{bilinear form relations}, we obtain
\begin{align*}
    &\quad{-\langle A,B\rangle} {+2\langle N,L\rangle}\\&=\langle M,M\rangle-\langle A,A\rangle-\langle B,B\rangle+\langle B,A\rangle {-2\langle L,N\rangle}\\&=\langle M,M\rangle{-\langle A,A\rangle-\langle B,B\rangle+\langle L,I,N\rangle}.
\end{align*}
Then by comparing with \eqref{F tilde} and \eqref{F hat 2}, the result follows immediately.
\end{proof}


\begin{proposition}\label{iHall algebra formula}
For any $A,B\in\mathcal{A}$ and $\alpha, \beta \in K_0(\mathcal{A})$, the following equation holds in $\iH(\ca)$:
$$([A]* [K_\alpha])*([B]*[K_\beta])=\sum\limits_{[M]}\widetilde{F}_{AB}^M\cdot[M]*
[K_{\frac{\widehat{A}+\widehat{B}-\widehat{M}}{2}+\alpha+\beta}].$$
\end{proposition}

\begin{proof}
By \cite[Proposition 3.5]{CLR}, in $\iH(\ca)$ we have
$$[A]*[B]=\sum_{[M]} \sum_{[L],[I],[N]} \sqq^{-\langle A,B\rangle} q^{\langle N,L\rangle}\cdot\frac{ a_L a_Ia_N}{a_M }\cdot F_{LI}^BF_{NL}^MF_{IN}^A		\cdot [M]*[K_{\widehat{A}-\widehat{N}}].$$
In case $F_{LI}^BF_{NL}^MF_{IN}^A\neq0$, there is a three-cycle \eqref{three cycle}. Then by Lemma \ref{lem4.5}, we obtain $\widehat{A}-\widehat{N}=\dfrac{\widehat{A}+\widehat{B}-\widehat{M}}{2}$, which is independent of $[L],[I],[N]$. Hence in $\iH(\ca)$ we have
$$[A]*[B]=\sum\limits_{[M]}\widetilde{F}_{AB}^M\cdot[M]*
[K_{\frac{\widehat{A}+\widehat{B}-\widehat{M}}{2}}].$$
Recall that $[K_{\alpha}]\;(\alpha\in K_0(\mathcal{A}))$ are central in $\iH(\ca)$.
Then we are done.
\end{proof}}


%

Now we can state the main result of this subsection.
\begin{proposition}
\label{prop4.8}
There is an algebra isomorphism
\begin{align*}
\widetilde{\Xi}_\ca:
\iH(\ca)
&\longrightarrow \widetilde{\mathcal{H}}_{\Delta}(\mathcal{A}).
\\
[M]*[K_\alpha]
&\mapsto \sqq^{-\langle M,M\rangle}[M][K_{\alpha}]
\end{align*}
\end{proposition}

\begin{proof}
For any $A,B\in\mathcal{A}$ and $\alpha, \beta \in K_0(\mathcal{A})$, by Proposition \ref{iHall algebra formula} we have
 \begin{align*}
    &\quad\widetilde{\Xi}_\ca\big(([A]* [K_\alpha])*([B]* [K_\beta])\big)\\&=
   \sum_{[M]}\ {\widetilde F}_{AB}^M	\cdot \widetilde{\Xi}_\ca([M]*[K_{\frac{\widehat{A}+\widehat{B}-\widehat{M}}{2}+\alpha+\beta}])\\&=
   \sum_{[M]}\ \sqq^{-\langle M,M\rangle}	{\widetilde F}_{AB}^M	\cdot [M][K_{\frac{\widehat{A}+\widehat{B}-\widehat{M}}{2}+\alpha+\beta}].
\end{align*}
On the other hand, \begin{align*}
    &\quad\widetilde{\Xi}_\ca([A]* [K_\alpha])*\widetilde{\Xi}_\ca([B]* [K_\beta])\\&=
    (\sqq^{-\langle A,A\rangle}[A][K_{\alpha}])*(\sqq^{-\langle B,B\rangle}[B][K_{\beta}])\\&=
    \sqq^{-\langle A,A\rangle-\langle B,B\rangle}
    {\sum_{[M]}\widehat{F}_{AB}^M\cdot } [M][K_{\frac{\widehat{A}+\widehat{B}-\widehat{M}}{2}+\alpha+\beta}].
\end{align*}
It follows from Lemma \ref{tilde and hat} that $$ \widetilde{\Xi}_\ca([A]* [K_\alpha])*\widetilde{\Xi}_\ca([B]* [K_\beta])=\widetilde{\Xi}_\ca\big(([A]* [K_\alpha])*([B]* [K_\beta])\big).$$ Hence $\widetilde{\Xi}_\ca$ is an algebra homomorphism. By comparing the basis of $\iH(\ca)$ and of $\widetilde{\mathcal{H}}_{\Delta}(\mathcal{A})$, we know that $\widetilde{\Xi}_\ca$ is bijective, so it is an algebra isomorphism.
\end{proof}

{

As a consequence of Propositions \ref{prop4.8} and \ref{prop3.2}, we obtain

\begin{proposition}
\label{cor4.6}
	There exists an algebra epimorphism
\[
\widetilde{\Phi}: \widetilde{\mathcal{H}}_{\Delta}(\ca)\longrightarrow \mathcal{H}_{\Delta}(\mathcal{A}),
\]
\[
 [M][K_\alpha]\longmapsto [M]
\]
with ${\rm Ker}\ \widetilde{\Phi}=\langle [K_\alpha]-1,\alpha\in K_0(\mathcal{A})\rangle$.
\end{proposition}

\begin{proof}
According to \cite[Theorem 3.4]{CLR}, there exists an algebra epimorphism $$\Phi:\iH(\ca)\longrightarrow\mathcal{DH}_1(\mathcal{A});\quad [M]*[K_\alpha]\mapsto
\sqrt{\{M,M\}}\cdot\tilde{a}_M\cdot u_{[M]},$$
with ${\rm Ker}\ \Phi=\langle [K_\alpha]-1,\alpha\in K_0(\mathcal{A})\rangle$.
By Propositions \ref{prop4.8} and \ref{prop3.2}, there are isomorphisms $$\widetilde{\Xi}_\ca:\iH(\ca)
\to\widetilde{\mathcal{H}}_{\Delta}(\mathcal{A})\text{\quad and\quad} \Xi_\ca:\mathcal{DH}_1(\mathcal{A})
\to \mathcal{H}_{\Delta}(\mathcal{A}).$$ Let $\widetilde{\Phi}$ be the composition $$\widetilde{\Phi}=\Xi_\ca\circ\Phi\circ\widetilde{\Xi}_\ca^{-1}: \widetilde{\mathcal{H}}_{\Delta}(\mathcal{A})
\longrightarrow\iH(\ca)\longrightarrow\mathcal{DH}_1(\mathcal{A})
\rightarrow \mathcal{H}_{\Delta}(\mathcal{A}).$$ Therefore, {$\widetilde{\Phi}([M][K_\alpha])=[M]$} for any $M\in\mathcal{A}$ and $\alpha\in K_0(\mathcal{A})$. In particular, we have the following commutative diagram:
\[
\begin{tikzpicture}
\node (2) at (0,0) {${\rm Ker}\ {\Phi}$};
\node (3) at (0,-3){${\rm Ker}\ \widetilde{\Phi}$};
\node (5) at (5,0){$\iH(\ca)$};
\node (6) at (5,-3) {$\widetilde{\mathcal{H}}_{\Delta}(\ca)$};
\node (8) at (10,0){$\cd\ch_1(\ca)$};
\node (9) at (10,-3) {$\mathcal{H}_{\Delta}(\ca)$.};
\draw[->] (2) --node[right  ]{$\cong$} (3);
\draw[->] (5) --node[right ]{$\cong$} (6);
\draw[->] (8) --node[right ]{$\cong$} (9);
\draw[>->] (2) --node[left]{} (5);
\draw[>->] (3) --node[left]{} (6);
\draw[<<-] (8) --node[below right]{} (5);
\draw[<<-] (9) --node[left]{} (6);
\node (6) at (4.7,-1.5) {$\widetilde{\Xi}_\ca$};
\node (6) at (5+4.7,-1.5) {${\Xi}_\ca$};
\node (6) at (7.5,.3-3) { $\widetilde{\Phi}$\quad };
\node (6) at (7.5,.3) { ${\Phi}$\quad };
\end{tikzpicture}
\]
It follows that ${\rm Ker}\ \widetilde{\Phi}={\widetilde{\Xi}}_\ca({\rm Ker}\ {\Phi})=\langle [K_\alpha]-1,\alpha\in K_0(\mathcal{A})\rangle$. We are done.
\end{proof}}

\subsection{New realization of universal $\imath$quantum groups}

Let $Q$ be an arbitrary quiver without loops. The simple $\mathbf{k} Q$-modules are denoted by $S_i$ for $i\in \I=Q_0$. We consider the hereditary abelian category $\mathcal{A}={\rm rep^{nil}}(\mathbf{k} Q)$.
%
Denote by $\iH(\bfk Q):=\iH(\ca)$, $\widetilde{\mathcal{H}}_{\Delta}(\bfk Q):=\widetilde{\mathcal{H}}_{\Delta}(\ca)$ and  $\widetilde{\Xi}_Q:=\widetilde{\Xi}_{\ca}$. The following result states that we can realize the universal $\imath$quantum groups via extended $\Delta$-Hall algebras.

\begin{theorem}
\label{thm4.11} Let $Q$ be an arbitrary quiver without loops.
	 There exists a $\Q(\sqq)$-algebra embedding
	\begin{align*}
		\widetilde{\Theta}_Q: \tUi_{|v= \sqq} &\longrightarrow \widetilde{\mathcal{H}}_{\Delta}(\bfk Q),
	\end{align*}
	which sends
	\begin{align*}
		B_i \mapsto \frac{-1}{q-1}\cdot\sqq^{-\langle S_i,S_i\rangle}\cdot[S_{i}],
		&\qquad
		\tk_i \mapsto - q^{-1}[K_i], \qquad\text{ for }i \in \I.
	\end{align*}
In particular, if $Q$ is a Dynkin quiver, the embedding $\widetilde{\Theta}_Q$ is precisely an isomorphism.
\end{theorem}

\begin{proof}
Recall from \cite{LW19a,LW20a,LS21} that the $\imath$Hall algebra $\iH(\bfk Q)$ of ${\rm rep^{nil}}(\mathbf{k} Q)$ provides a realization of the universal $\imath$quantum group $\tUi$ associated to $Q$. Namely, there is an embedding $$\widetilde{\psi}_Q: \tUi_{|v= \sqq}\longrightarrow \iH(\bfk Q),$$ which sends
	\begin{align*}
		B_i \mapsto \frac{-1}{q-1}[S_{i}],
		&\qquad
		\tk_i \mapsto - q^{-1}[K_i], \qquad\text{ for }i \in \I.
	\end{align*}
In particular, if $Q$ is a Dynkin quiver, the embedding $\widetilde{\psi}_Q$ is precisely an isomorphism.

On the other hand, by Proposition \ref{prop4.8}, there exists an algebra isomorphism $$\widetilde{\Xi}_Q:
\iH(\bfk Q)
\rightarrow \widetilde{\mathcal{H}}_{\Delta}(\bfk Q),$$ which sends\[
[M]*[K_\alpha]
\mapsto \sqq^{-\langle M,M\rangle}[M][K_\alpha];\quad \forall M\in\mathcal{A},\;\alpha\in K_0(\mathcal{A}).
\]
Set $\widetilde{\Theta}_Q=\widetilde{\Xi}_Q\circ \widetilde{\psi}_Q$, then we are done.
\end{proof}

%
%

\section{Twisting on extended $\Delta$-Hall algebras}

In this section, we consider the twisting on the extended $\Delta$-Hall algebra $\widetilde{\mathcal{H}}_{\Delta}(\ca)$, in order to recover the semi-derived Hall algebra $\mathcal{SDH}(\mathcal{C}_1(\ca))$ of $\mathcal{C}_1(\ca)$. Moreover, by extending $K_0(\ca)$ to $K_0(\ca)\otimes_{\mathbb{Z}} \dfrac{1}{2}\mathbb{Z}$, we establish an algebra isomorphism between the extended version of the $\Delta$-Hall algebra and the derived Hall algebra of $\ca$.

\subsection{Twisted extended $\Delta$-Hall algebras and semi-derived Hall algebras}


Let $\varphi:K_0(\ca)\times K_0(\ca)\rightarrow\mathbb{Q}(\sqq)$ be a multiplicative bilinear form on $K_0(\ca)$ in the following sense:
\[
\varphi(\alpha_1+\alpha_2,\beta)=\varphi(\alpha_1,\beta)\cdot\varphi(\alpha_1,\beta)\quad {\text {and} \quad}\varphi(\alpha,\beta_1+\beta_2)=\varphi(\alpha,\beta_1)\cdot\varphi(\alpha,\beta_2)\]
for $\alpha,\alpha_1,\alpha_2,\beta,\beta_1,\beta_2\in K_0(\ca)$.

\begin{proposition}\label{twist}
For any multiplicative bilinear form $\varphi:K_0(\ca)\times K_0(\ca)\rightarrow\mathbb{Q}(\sqq)$, the following multiplication
\[
[A] [K_\alpha]\divideontimes [B][K_\beta]= \varphi(\widehat{A}+2\alpha,\widehat{B}+2\beta)\cdot[A] [K_\alpha]* [B][K_\beta]
\]
defines an associative algebra structure on $\widetilde{\mathcal{H}}_{\Delta}(\ca)$, called the \emph{twisted extended $\Delta$-Hall algebra} of $\ca$ via $\varphi$, and denoted by
${}_{\varphi}\widetilde{\mathcal{H}}_{\Delta}(\ca)$.
\end{proposition}

\begin{proof}
For any $[A], [B], [C]\in {\rm Iso}(\mathcal{A})$ and $\alpha, \beta, \gamma \in K_0(\mathcal{A})$, we have
\begin{align*}
&\qquad\big([A] [K_\alpha]\divideontimes [B][K_\beta]\big)\divideontimes [C][K_\gamma]\\&=
\varphi(\widehat{A}+2\alpha,\widehat{B}+2\beta)\cdot\big([A] [K_\alpha]*[B][K_\beta]\big)\divideontimes [C][K_\gamma]\\&=\varphi(\widehat{A}+2\alpha,\widehat{B}+2\beta)
\varphi(\widehat{A}+2\alpha+\widehat{B}+2\beta,\widehat{C}+2\gamma)\cdot\big([A] [K_\alpha]*[B][K_\beta]\big)* [C][K_\gamma]\\&=\varphi(\widehat{A}+2\alpha,\widehat{B}+2\beta)
\varphi(\widehat{A}+2\alpha,\widehat{C}+2\gamma)\varphi(\widehat{B}+2\beta,\widehat{C}+2\gamma)
\cdot\big([A] [K_\alpha]*[B][K_\beta]\big)* [C][K_\gamma].
\end{align*}
On the other hand,
\begin{align*}
&\qquad[A] [K_\alpha]\divideontimes \big([B][K_\beta]\divideontimes [C][K_\gamma]\big)\\&=
\varphi(\widehat{B}+2\beta,\widehat{C}+2\gamma)\cdot[A] [K_\alpha]\divideontimes\big([B][K_\beta]* [C][K_\gamma]\big)\\&=\varphi(\widehat{B}+2\beta,\widehat{C}+2\gamma)\varphi(\widehat{A}+2\alpha,\widehat{B}+2\beta+\widehat{C}+2\gamma)\cdot[A] [K_\alpha]*\big([B][K_\beta]* [C][K_\gamma]\big)\\&=\varphi(\widehat{A}+2\alpha,\widehat{B}+2\beta)\varphi(\widehat{A}+2\alpha,\widehat{C}+2\gamma)\varphi(\widehat{B}+2\beta,\widehat{C}+2\gamma)\cdot[A] [K_\alpha]*\big([B][K_\beta]* [C][K_\gamma])\big.
\end{align*}
Hence, by the associativity of $\widetilde{\mathcal{H}}_{\Delta}(\ca)$ (see Proposition \ref{extended version}) we obtain $$\big([A] [K_\alpha]\divideontimes [B][K_\beta]\big)\divideontimes [C][K_\gamma]=[A] [K_\alpha]\divideontimes \big([B][K_\beta]\divideontimes [C][K_\gamma]\big).$$ Then we are done.
\end{proof}

As an application, we find that the semi-derived Hall algebra $\mathcal{SDH}(\mathcal{C}_1(\ca))$ of $\mathcal{C}_1(\ca)$,
can be obtained from the extended $\Delta$-Hall algebra by suitable twisting.

\begin{proposition} Let $\varphi:K_0(\ca)\times K_0(\ca)\rightarrow\mathbb{Q}(\sqq)$ be a multiplicative bilinear form given by {$\varphi(\alpha,\beta)=\sqq^{-\langle\alpha,\beta\rangle}$}. Then there is an algebra isomorphism:\begin{align*}
_\varphi\widetilde{\Xi}_\ca: \mathcal{SDH}(\mathcal{C}_1(\ca))&\longrightarrow{}_{\varphi}\widetilde{\mathcal{H}}_{\Delta}(\ca).
\\
[M]\diamond[K_\alpha]
&\mapsto \sqq^{-\langle M,M\rangle}[M][K_{\alpha}]
\end{align*}
\end{proposition}

\begin{proof}

{Recall that the basis of $\mathcal{SDH}(\mathcal{C}_1(\ca))$ is given by
$$
\big\{[M]\diamond [K_\alpha]\mid [M]\in{\rm Iso}(\mathcal{A}), {\alpha\in K_0(\mathcal{A})}\big\},
$$
and according to (\ref{eq:tH}) and Proposition \ref{iHall algebra formula}, its multiplication is given by \[
([A]\diamond [K_\alpha])\diamond ([B]\diamond[K_\beta])= \sqq^{-\langle\widehat{A}+2\alpha,\widehat{B}+2\beta\rangle}\sum_{[M]}\widetilde{F}_{AB}^M\cdot[M]\diamond
[K_{\frac{\widehat{A}+\widehat{B}-\widehat{M}}{2}+\alpha+\beta}].
\]}

On the other hand, the basis of ${}_{\varphi}\widetilde{\mathcal{H}}_{\Delta}(\ca)$ is given by
$$
\big\{[M][K_\alpha]\mid [M]\in{\rm Iso}(\mathcal{A}), {\alpha\in K_0(\mathcal{A})}\big\},
$$
and according to Proposition \ref{twist}, its multiplication is given by \[
[A] [K_\alpha]\divideontimes [B][K_\beta]= \sqq^{-\langle\widehat{A}+2\alpha,\widehat{B}+2\beta\rangle}\sum_{[M]}\widehat{F}_{AB}^M\cdot[M]
[K_{\frac{\widehat{A}+\widehat{B}-\widehat{M}}{2}+\alpha+\beta}].
\]


Therefore, by Lemma \ref{tilde and hat},
\begin{align*}
&\quad _\varphi\widetilde{\Xi}_\ca\big(([A]\diamond[K_\alpha])\diamond([B]\diamond[K_\beta])\big)\\
=&\sqq^{-\langle\widehat{A}+2\alpha,\widehat{B}+2\beta\rangle-\langle M,M\rangle}\sum_{[M]}\widetilde{F}_{AB}^M\cdot[M]
[K_{\frac{\widehat{A}+\widehat{B}-\widehat{M}}{2}+\alpha+\beta}]
\\=&\sqq^{-\langle\widehat{A}+2\alpha,\widehat{B}+2\beta\rangle-\langle A,A\rangle-\langle B,B\rangle}\sum_{[M]}\widehat{F}_{AB}^M\cdot[M]
[K_{\frac{\widehat{A}+\widehat{B}-\widehat{M}}{2}+\alpha+\beta}]\\
=&\sqq^{-\langle A,A\rangle-\langle B,B\rangle} ([A] [K_\alpha]\divideontimes [B][K_\beta])\\
=& {}_\varphi\widetilde{\Xi}_\ca([A]\diamond[K_\alpha])\divideontimes {}_\varphi\widetilde{\Xi}_\ca([B]\diamond[K_\beta]).
\end{align*}
Hence, $_\varphi\widetilde{\Xi}_\ca$ is an algebra homomorphism. Obviously, it is an isomorphism of $\mathbb{Q}(\sqq)$-vector spaces. Then we are done.
\end{proof}

\subsection{Extended $\Delta$-Hall algebras and extended derived Hall algebras}
In this subsection, we consider a new kind of extension of the $\Delta$-Hall algebra ${\mathcal{H}}_{\Delta}(\mathcal{A})$, with the aim to establish an \emph{algebraic isomorphism} with certain extended version of the derived Hall algebra $\mathcal{DH}_1(\ca)$.

Recall that the Grothendieck group $K_0(\ca)$ is free. Consider the $\mathbb{Z}$-module $$\dfrac{1}{2}K_0(\ca):=K_0(\ca)\otimes_{\mathbb{Z}} \dfrac{1}{2}\mathbb{Z}=\{\frac{\beta}{2}\mid \beta\in K_0(\ca)\}.$$
Clearly, $K_0(\ca)$ is a subgroup of $\dfrac{1}{2}K_0(\ca)$.

Denote by $\widetilde{\widetilde{\mathcal{H}}}_{\Delta}(\mathcal{A})$ the $\mathbb{Q}(\sqq)$-vector space with the basis
$$
\big\{[M] [K_\alpha]\mid [M]\in{\rm Iso}(\mathcal{A}), {\alpha\in \dfrac{1}{2}K_0(\mathcal{A})}\big\}.
$$
Similar as Proposition \ref{extended version}, the following result holds.
\begin{proposition}
The $\mathbb{Q}(\sqq)$-vector space $\widetilde{\widetilde{\mathcal{H}}}_{\Delta}(\mathcal{A})$ endowed with the multiplication defined by
\[[A] [K_\alpha]*[B][K_\beta]=\sum_{[M]}\widehat{F}_{AB}^M\cdot[M]
[K_{\frac{\widehat{A}+\widehat{B}-\widehat{M}}{2}+\alpha+\beta}],\]
forms an associative algebra with unit $[0][K_0]$.
\end{proposition}

Consider the tensor product space $\mathcal{H}_{\Delta}(\mathcal{A})\otimes_{\mathbb{Q}} \mathbb{Q}{(\dfrac{1}{2}K_0(\ca))}$, equipped with the standard multiplication, i.e.,
{$(x\otimes u)\ast(y\otimes w)=(x\ast y)\otimes (u\ast w)$}. It turns out that the tensor product algebra is isomorphic to the extended $\Delta$-Hall algebra $\widetilde{\widetilde{\mathcal{H}}}_{\Delta}(\mathcal{A})$.


\begin{proposition}\label{tensor space}
There is an algebra isomorphism
{\begin{align*}
\Phi_{\otimes}: \widetilde{\widetilde{\mathcal{H}}}_{\Delta}(\mathcal{A})&\longrightarrow \mathcal{H}_{\Delta}(\mathcal{A})\otimes_{\mathbb{Q}} \mathbb{Q}(\dfrac{1}{2}K_0(\ca)).
\\
[M][K_\alpha]&\mapsto [M]\otimes [K_{\frac{\widehat{M}}{2}+\alpha}]
\end{align*}}
\end{proposition}

\begin{proof}
For any $[A], [B]\in {\rm Iso}(\mathcal{A})$ and $\alpha, \beta \in \dfrac{1}{2}K_0(\mathcal{A})$, we have \begin{align*}
    &\quad\Phi_{\otimes}([A][K_\alpha]*[B][K_\beta])\\&=
    \Phi_{\otimes}\big(\sum_{[M]}\widehat{F}_{AB}^M\cdot [M][K_{\frac{\widehat{A}+\widehat{B}-\widehat{M}}{2}+\alpha+\beta}]\big)\\&=
    \sum_{[M]}\widehat{F}_{AB}^M\cdot [M]\otimes [K_{\frac{\widehat{M}}{2}+(\frac{\widehat{A}+\widehat{B}-\widehat{M}}{2}+\alpha+\beta)}]\\&=
    \sum_{[M]}\widehat{F}_{AB}^M\cdot [M]\otimes [K_{\frac{\widehat{A}+\widehat{B}}{2}+\alpha+\beta}],
\end{align*}
and \begin{align*}
    &\quad\Phi_{\otimes}([A][K_\alpha])*\Phi_{\otimes}([B][K_\beta])\\&=
    ([A]\otimes [K_{\frac{\widehat{A}}{2}+\alpha}])*([B]\otimes [K_{\frac{\widehat{B}}{2}+\beta}])\\&=
    \sum_{[M]}\widehat{F}_{AB}^M\cdot [M]\otimes [K_{\frac{\widehat{A}+\widehat{B}}{2}+\alpha+\beta}].
\end{align*}
It follows that $\Phi_{\otimes}$ is an algebra homomorphism.

Note that ${\mathcal{H}_{\Delta}}(\mathcal{A})\otimes_{\mathbb{Q}} \mathbb{Q}(\dfrac{1}{2}K_0(\ca))$ is a $\mathbb{Q}(\sqq)$-vector space with the basis $$\big\{[M]\otimes [K_\alpha]\mid[M]\in {\rm Iso}(\ca),\alpha\in \dfrac{1}{2}K_0(\ca)\big\}.$$ For each $[M]\otimes [K_\alpha]\in\mathcal{H}_{\Delta}(\mathcal{A})\otimes_{\mathbb{Q}} \mathbb{Q}(\dfrac{1}{2}K_0(\ca))$, there is an element $[M] [K_{\alpha-\frac{\widehat{M}}{2}}]\in\widetilde{\widetilde{\mathcal{H}}}_{\Delta}(\mathcal{A})$, such that $\Phi_{\otimes}([M] [K_{\alpha-\frac{\widehat{M}}{2}}])=[M]\otimes [K_\alpha]$. Hence $\Phi_{\otimes}$ is surjective. Clearly, $\Phi_{\otimes}$ is also injective. Therefore, $\Phi_{\otimes}$ is an algebra isomorphism.
\end{proof}

As an immediately consequence of Propositions \ref{prop3.2} and \ref{tensor space}, we obtain
\begin{corollary}
\label{cor4.10}
There is an algebra isomorphism
\begin{align*}
{\Psi}_{\otimes}: \widetilde{\widetilde{\mathcal{H}}}_{\Delta}(\mathcal{A})&\longrightarrow \mathcal{DH}_1(\ca)\otimes_{\mathbb{Q}} \mathbb{Q}(\dfrac{1}{2}K_0(\ca)).
\\
[M][K_\alpha]&\mapsto a_{M}\cdot u_{[M]}\otimes [K_{\frac{\widehat{M}}{2}+\alpha}]
\end{align*}
\end{corollary}

\begin{proof}
{Recall from Propositions \ref{prop3.2} and \ref{tensor space} that there are isomorphisms $\Xi_\ca$ and $\Phi_{\otimes}$ respectively. Consider the composition ${\Psi}_{\otimes}:=(\Xi_\ca^{-1}\otimes\mathrm{id})\circ\Phi_{\otimes}$ as follows: } $$\widetilde{\widetilde{\mathcal{H}}}_{\Delta}(\mathcal{A})\longrightarrow \mathcal{H}_{\Delta}(\mathcal{A})\otimes_{\mathbb{Q}} \mathbb{Q}(\dfrac{1}{2}K_0(\ca))
\longrightarrow\mathcal{DH}_1(\ca)\otimes_{\mathbb{Q}} \mathbb{Q}(\dfrac{1}{2}K_0(\ca)).$$ Then we are done.
\end{proof}

\begin{remark}
	{$(1)$ By definition, ${\widetilde{\mathcal{H}}}_{\Delta}(\mathcal{A})$ is a natural subalgebra of $\widetilde{\widetilde{\mathcal{H}}}_{\Delta}(\mathcal{A})$. Therefore, under the isomorphisms $\widetilde{\Xi}_\ca:
			\iH(\ca)
			\longrightarrow \widetilde{\mathcal{H}}_{\Delta}(\mathcal{A})$ and ${\Psi}_{\otimes}$, the $\imath$Hall algebra $\iH(\ca)$ is a subalgebra of $\mathcal{DH}_1(\ca)\otimes_{\mathbb{Q}} \mathbb{Q}(\dfrac{1}{2}K_0(\ca))$.
	
	$(2)$ The inverse of ${\Psi}_{\otimes}$ is given by:
	\begin{align*}
	{\Psi}_{\otimes}^{-1}: \;&\mathcal{DH}_1(\ca)\otimes_{\mathbb{Q}} \mathbb{Q}(\dfrac{1}{2}K_0(\ca))\longrightarrow \widetilde{\widetilde{\mathcal{H}}}_{\Delta}(\mathcal{A}).
	\\
	&  u_{[M]}\otimes [K_{\alpha}]\mapsto \frac{1}{a_M}\cdot [M][K_{-\frac{\widehat{M}}{2}+\alpha}]
	\end{align*}
Define the degree function on $\widetilde{\widetilde{\mathcal{H}}}_{\Delta}(\mathcal{A})$ by $$\deg([M][K_{\alpha}])=\widehat{M}+2\alpha;\quad \forall M\in\mathcal{A},\;\alpha\in K_0(\mathcal{A}).$$ Then the image of $\mathcal{DH}_1(\ca)\otimes 1$ under ${\Psi}_{\otimes}^{-1}$ has the expression $$\mathbb{Q}(\sqq)\{[M][K_{-\frac{\widehat{M}}{2}}]\mid [M]\in {\rm Iso}(\ca)\},$$ which is actually the subspace of $\widetilde{\widetilde{\mathcal{H}}}_{\Delta}(\mathcal{A})$ with degree $0$.

Therefore, the 1-periodic derived Hall algebra $\mathcal{DH}_1(\ca)$ can be embedded into
$\widetilde{\widetilde{\mathcal{H}}}_{\Delta}(\mathcal{A})$, which is an extended version of ${\widetilde{\mathcal{H}}}_{\Delta}(\mathcal{A})\cong \iH(\ca)$, hence embedded into an extended version of the twisted semi-derived Hall algebra.
In general, for $t$-periodic cases with $t\geq 1$, the embeddings from the derived Hall algebras to certain extended version of twisted semi-derived Hall algebras also hold, see Lin-Peng \cite[Theorem 5.6]{LP2}.
}
\end{remark}

\noindent{\bf Acknowledgment.}
This work was supported by the National Natural Science Foundation of China (No. 11871404, 11801473) and the
Fundamental Research Funds for Central Universities of China (No. 20720220043).


\end{document}